\documentclass[a4paper,11pt]{article}
\usepackage[utf8]{inputenc}
\usepackage[T1]{fontenc}
\usepackage[english]{babel}
\usepackage{a4wide}

\usepackage[]{hyperref}
\hypersetup{
    colorlinks=true,       
    linkcolor=red,          
    citecolor=blue,        
    filecolor=magenta,      
    urlcolor=cyan           
}

\usepackage[leqno]{amsmath}
\usepackage{amssymb}
\usepackage{mathrsfs}
\usepackage{amsthm}
\usepackage{amsxtra}
\usepackage{bm}
\usepackage[titletoc]{appendix}

\usepackage{graphicx}

\theoremstyle{plain}
\newtheorem{thm}{Theorem}[section]
\newtheorem{prop}[thm]{Proposition}
\newtheorem{lem}[thm]{Lemma}

\theoremstyle{definition}
\newtheorem{defn}[thm]{Definition}
\newtheorem{nota}[thm]{Notation}

\theoremstyle{remark}
\newtheorem*{rem}{Remark}

\numberwithin{equation}{section}

\title{Blow-up conditions for gravity water-waves}
\author{Thibault de Poyferr\'{e}\footnote{UMR 8553 CNRS, Laboratoire de Mathématiques et Applications de l'Ecole Normale Supérieure, 75005 Paris, France. Email: tdepoyfe@dma.ens.fr}}
\date{}

\DeclareMathOperator{\RE}{Re}
\DeclareMathOperator{\rot}{rot}

\DeclareMathOperator{\dive}{div}

\def \div {\dive}
\def\C{\bm{\mathrm{C}}}
\def\d{\,\mathrm{d}}
\def\eps{\varepsilon}

\def\jap{\langle D_x\rangle}
\def\la{\left\vert}
\def\lA{\left\Vert}
\def\lb{\left[}
\def\lB{\left\{}
\def\lp{\left(}
\def\ls{\left\langle}

\def\mez{\frac{1}{2}}
\def\M{\mathcal{M}}
\def\N{\bm{\mathrm{N}}}
\def\nab{\nabla_{\!x,y}}
\def\nabz{\nabla_{\!x,z}}
\def\P{\mathcal{P}}
\def\ph{\varphi}
\def\R{\bm{\mathrm{R}}}
\def\ra{\right\vert}
\def\rA{\right\Vert}
\def\rb{\right]}
\def\rB{\right\}}
\def\rp{\right)}
\def\rs{\right\rangle}

\def\tder{\frac{\mathrm{d}}{\mathrm{d}t}}

\begin{document}

\maketitle

\begin{abstract}
We exhibit blow-up conditions for the gravity water-waves equations in any dimension and in domains with arbitrary bottoms. 
We follow the method by Alazard, Burq and Zuily of using a paradifferential reduction of the equations and derive precise a priori Sobolev estimates. Those estimates are then used to prove three different blow-up conditions where neither the boundedness of the curvature of the surface nor 
the boundedness in time of the Lipschitz norm of the velocity are needed.
\end{abstract}

\section{Introduction}

In this paper, we derive a blow-up criterion for the water-waves system, without surface tension and with arbitrary bottom.
The water-waves problem is the study of the motion under the influence of gravity of a homogeneous, inviscid fluid, typically water, inside a laterally infinite container, and separated from the atmosphere by 
a free interface. 

We will assume the presence of a constant gravity field acting along the~$e_y$ axis, distinguishing it from the horizontal plane. This horizontal plane will be of dimension~$d\geq1$, with in 
applications~$d=1$ or~$2$. Positions will be expressed in coordinates~$(x,y)\in\R^d\times\R$.
We write~$\nabla=\nabla_{\!x}=(\partial_{x_1},\dots,\partial_{x_d})$ and~$\nab=(\nabla_{\!x},\partial_y)$.

At each time~$t\in\R^+$, the fluid will occupy a domain~$\Omega(t)$. We suppose that the free surface, which will be denoted~$\Sigma(t)$ is the graph of a continuous function~$y=\eta(t,x)$ representing the variation of the water 
surface from its rest level.
In order to account for a wide variety of bottoms, we will consider a simply connected open subset~$\mathcal{O}$ of~$\R^{d+1}$, such that
\begin{equation*}
	\Omega(t)=\lB(x,y)\in\R^d\times\R;\;(x,y)\in\mathcal{O},y<\eta(t,x)\rB.
\end{equation*}
We suppose that there exists $h>0$ such that, for all times, the domain~$\Omega(t)$ contains a horizontal strip of width~$h$,
\begin{equation} \label{eq:bothyp}
	\Omega_h(t):=\lB(x,y)\in\R^d\times\R;\;\eta(t,x)-h<y<\eta(t,x)\rB\subset\Omega(t).
\end{equation}
This means that the bottom, denoted by~$\Gamma$, is nowhere emerging (which precludes islands and beaches).

The velocity~$v(t,x,y)\in\R^{d+1}$ of the fluid occupying~$\Omega(t)$ follows the incompressible Euler equations
\begin{equation} \label{eq:Euler}
	\lB
	\begin{gathered}
		\partial_tv+\lp v\cdot\nab\rp v+\nab P=-ge_y,\\
		\div_{x,y}v=0,
	\end{gathered}
	\right. 
\end{equation}
where~$g$ is the acceleration of gravity, supposed constant and positive, and where~$P(t,x,y)\in\R$ is the pressure of the fluid.
It is customary in oceanography to impose in addition for the fluid to be curl-free, so that~$\rot_{x,y}v=0$ in~$\Omega(t)$.

In addition, we need to impose boundary conditions on~$\Sigma(t)$ and~$\Gamma$.
First there are the kinematic conditions that the fluid does not cross or leaves those boundaries, so that
\begin{gather} 
	v\cdot n=0\quad\text{on }\Gamma, \label{eq:BotCnd}\\
	\partial_t\eta=\sqrt{1+\la\nabla\eta\ra^2}\, v\cdot\nu\quad\text{on }\Sigma, \label{eq:SurfCnd}
\end{gather}
where~$n$ and~$\nu(t)$ are the the exterior unit normals respectively to~$\Gamma$ and~$\Sigma(t)$.
At last there is a dynamic boundary condition on the pressure. We suppose that there is no surface tension at the surface, which implies that there is no pressure jump between the fluid and the atmosphere.
We assume this atmospheric pressure to be constant, and we can change the definition of~$P$ by an additive constant to take~$P_\text{atm}=0$. Then
\begin{equation*}
	P\rvert_{y=\eta}=0.
\end{equation*}

\begin{rem}
	\begin{itemize}
		\item By imposing for the surface~$\Sigma(t)$ to be a graph, we implicitly assumed that our solutions will blow up when this ceases to be the case. 
		      It has been proved by Castro, C{\'o}rdoba, Fefferman, Gancedo and G{\'o}mez-Serrano~\cite{CastroCordobaFeffermanGancedoGomezSerranoSingGrav} 
		      (see also Coutand-Shkoller \cite{CoutandShkollerCMP}) that some cases of a non-graph smooth surface 
		      can evolve to a self-intersecting surface, the so-called splash singularities, where this physical model does not make sense anymore. 
		      This shows that any study of blow-up without the graph hypothesis should involve some geometric quantities.
		\item The curl-free hypothesis is a good approximation for most deep-ocean applications, however it ceases to apply near a cost or when we take the Coriolis effect into account.
		      See Castro and Lannes~(\cite{CastroLannesVortForm}) for a formulation and some results with vorticity.
		\item Since~$\Gamma$ is not always smooth, its normal may not be defined. We will later give a variational meaning to this condition, coinciding with the strong sense when the normal exists.
	\end{itemize}
\end{rem}
For more on those hypotheses and on this model see the book by Lannes~\cite{LannesWWPblm}.

Now from the simple connectedness of~$\mathcal{O}$, and therefore of~$\Omega(t)$, and because~$\div_{x,y}v=0$ and~$\rot_{x,y}v=0$,
we see that there exists a scalar function~$\phi$ defined on the fluid domain such that
\begin{gather*}
	\nab\phi=v\quad\text{in }\Omega,\\
	\Delta_{x,y}\phi=0\quad\text{in }\Omega.
\end{gather*}
Now the Euler equation~(\ref{eq:Euler}) and the boundary conditions~(\ref{eq:BotCnd}) and~(\ref{eq:SurfCnd}) can be recast for this velocity potential, 
becoming ---up to a harmless change of the definition of~$\phi$ by a time-dependent constant--- the Bernoulli formulation
\begin{equation} \label{eq:Bern}
	\lB
	\begin{aligned}
		\partial_t\phi+\mez\la\nab\phi\ra^2+gy=-P\quad&\text{in }\Omega(t),\\
		\partial_t\eta=\partial_y\phi-\nabla\eta\cdot\nabla\phi\quad&\text{on }\Sigma(t),\\
		\partial_n\phi=0\quad&\text{on }\Gamma,\\
		P=0\quad&\text{on }\Sigma(t).
	\end{aligned}
	\right.
\end{equation}

The Cauchy problem for this system has been widely studied, starting from the works of Nalimov~(\cite{NalimovCPProb}), Shinbrot~(\cite{ShinbrotInPbl}), Yosihara~(\cite{YosiharaGrWa}) 
and Craig~(\cite{CraigExWWBoussKdV}).
The first results for Sobolev spaces and without smallness assumptions are due to Wu~(\cite{WuWPos2D,WuWPos3D}). A recent extension with rougher data, essentially H\"{o}lder with exponent~$3/2$
has been proposed by Alazard, Burq and Zuily in~\cite{AlazardBurqZuilyExw/ST}, with another extension using Strichartz estimates in~\cite{AlazardBurqZuilyStrEstGrav}.
More Recently, Kinsley and Wu have obtained in~\cite{KinsleyWuAngled} a priori estimates covering the case with angled crests.
The next natural objective is 
to find a blow-up criterion for the system. 
Christodoulou and Lindblad~(\cite{ChristodoulouLindbladMotFrSur}) proved such a criterion involving geometric quantities for the case 
without bottom. 
They showed that the  solutions can be extended as long as the curvature of the surface and the derivative of the velocity remain bounded. More recently, Wang and Zhang~(\cite{WangZhangBrkWW})
used some of the methods of~\cite{AlazardBurqZuilyExw/ST} to prove that as long as
\begin{equation*}
	\sup_{0\leq t< T}\lA\kappa(t)\rA_{L^2\cap L^p}+\int_0^T\lA\lp\nabla V,\nabla B\rp\rA_{W^{1,\infty}}^6\d t
\end{equation*}
is bounded, the solution can be extended after the time~$T$. Here~$\kappa$ is the curvature of~$\Sigma$, $V$ and~$B$ are respectively the horizontal and vertical traces of the velocity~$v$ at~$\Sigma$
and~$p>2d$. 
We will prove three blow-up criterions which extend this result. The results proved in this paper
involve less regular norms of the free surface and are valid for the case with rough bottom. More importantly, we will prove two results which 
involve only $L^1$ norms in time of the highest-order norms. 
Notice that one of the results below (see Theorem~\ref{thm:mainsob2}) is used in \cite{AlazardBurqZuilyStrEstGrav}(see Section 5.4) to deduce an existence result 
from a priori Sobolev and Strichartz estimates. Since Strichartz estimates involve $L^2$ norms in time (in dimension $d\geq 2$), it is crucial 
to have a blow-up result which involves only $L^p$ norms for $p\leq 2$. In this direction, 
we will obtain sharp results involving only $L^1$ norms (see Theorem~\ref{thm:mainsob} and Theorem~\ref{thm:mainsob2}).
In the case of 2D water-waves ($d=1$), Hunter, Ifrim and Tataru have obtained in~\cite{HunterIfrimTataru2DWW} a blow-up criterion in holomorphic coordinates,
corresponding to ours but sharpened to BMO norms instead of Hölder norms in space.

An important quantity appears in the analysis of the system~(\ref{eq:Bern}), the so-called Rayleigh-Taylor coefficient
\begin{equation*}
	a:=-\partial_yP\rvert_{y=\eta}.
\end{equation*}
In order to solve the Cauchy problem, we need to make a positivity hypothesis on~$a$.
One of the important contributions of Wu's articles~\cite{WuWPos2D,WuWPos3D} is that this condition is always true when the depth is infinite, which corresponds to the case~$\Gamma=\emptyset$.
Lannes then proved the same result for a small regular perturbation of a flat bottom~(\cite{LannesWellPos}).

Inspired by Craig \cite{CraigExWWBoussKdV} and Lannes~\cite{LannesWellPos}, we 
will use the eulerian formulation of the equations in connection with elliptic estimates and 
microlocal or harmonic analysis. In particular, we use the Craig-Sulem-Zakharov formulation of the equations~(\cite{CraigSulemNumWW,ZakharovStabPer}). Notice that since the potential~$\phi$ is harmonic, it is entirely determined by its value at the surface. We define
\begin{equation*}
	\psi(t,x)=\phi(t,x,\eta(x)).
\end{equation*}
The equation can then be recast in terms of~$\eta$ and~$\psi$, which are functions defined on~$\R^d$.
In order to simplify the presentation, Craig and Sulem introduced the use of the Dirichlet-Neumann operator in~\cite{CraigSulemNumWW}. 
This operator is defined as associating to a function defined on~$\Sigma$ the exterior normal of its harmonic 
extension to~$\Omega(t)$. Here for convenience we re-normalize it to get
\begin{equation*}
	G(\eta)\psi=\sqrt{1+\la\nabla\eta\ra^2}\partial_n\phi\rvert_{z=\eta}.
\end{equation*}
With this operator, we get a closed system of equations, known as the Craig-Sulem-Zakharov System
\begin{equation} \label{eq:ZaCrSu}
\lB
\begin{aligned}
&		\partial_t\eta-G(\eta)\psi=0,\\
&		\partial_t\psi+g\eta+\mez\la\nabla\psi\ra^2-\mez\frac{\lp\nabla\eta\cdot\nabla\psi+G(\eta)\psi\rp^2}{1+\la\nabla\eta\ra^2}=0.
\end{aligned}
	\right.
\end{equation}
\begin{rem}
	\begin{itemize}
		\item Under this formulation, the system is Hamiltonian. This is what motivated the original idea of Zakharov. The Hamiltonian is 
		      \begin{equation} \label{eq:Ham}
		      	\mez\int_{\R^d}\psi G(\eta)\psi\d x+\mez\int_{\R^d}g\eta^2\d x,
		      \end{equation}
		      and is conserved by the evolution (see e.g. \cite{LannesWWPblm}).
		\item The formal equivalence of this system to the original one is clear and we refer to~\cite{AlazardBurqZuilyZakEul} 
		for a rigorous proof.
	\end{itemize}
\end{rem}

This work is based upon the paper~\cite{AlazardBurqZuilyExw/ST} by Alazard, Burq and Zuily. To recall their main result, we introduce the vertical and horizontal parts of the velocity at the surface,
\begin{equation*}
	B:=(\partial_y\phi)\rvert_{y=\eta},\quad V=(\nabla_{\!x}\phi)\rvert_{y=\eta}.
\end{equation*}
Those quantities can be computed from knowing only~$\eta$ and~$\psi$.
Then
\begin{thm}[Theorem 2.1 of \cite{AlazardBurqZuilyExw/ST}] \label{thm:AlazardBurqZuilyExw/ST}
Let $d\geq1$, $s>1+d/2$ and consider $(\eta_0,\psi_0)$ such that 
\begin{enumerate}
  \item $\eta_0\in H^{s+\mez}(\R^d),\quad\psi_0\in H^{s+\mez}(\R^d),\quad V_0\in H^s(\R^d),\quad B_0\in H^s(\R^d),$
  \item there is $h>0$ such that condition (\ref{eq:bothyp}) holds for $t=0$,
  \item there is a positive constant $c$ such that, for any $x\in\R^d$, $a_0(x)\geq c$.
\end{enumerate}
Then there exists $T>0$ such that the Cauchy problem for (\ref{eq:ZaCrSu}) with initial data $(\eta_0,\psi_0)$ has a unique solution 
$(\eta,\psi)$ in $C^0\left([0,T];H^{s+\mez}(\R^d)\times H^{s+\mez}(\R^d)\right)$, such that
\begin{enumerate}
\item we have $(V,B)\in C^0\left([0,T];H^s(\R^d)\times H^s(\R^d)\right)$,
\item the condition (\ref{eq:bothyp}) holds for $0\leq t\leq T$, with $h$ replaced by $h/2$,
\item for any $0\leq t\leq T$ and any $x\in\R^d$, $a(t,x)\geq c/2$.
\end{enumerate}
\end{thm}
The proof of this theorem relies on paradifferential calculus to reduce the equations to a quasi-linear system, and then use classical energy methods for hyperbolic symmetrizable 
quasi-linear systems.
Some notions about paradifferential calculus are recalled in appendix~\ref{ap:para}. It has the advantage of yielding tame estimates of the various nonlinearities, meaning that those estimates are linear with respect to the 
higher order norm.
This will enable us to derive new a priori energy estimates for the paradifferential reduction of the system, from which we will derive a blow-up criterion complementing Theorem~\ref{thm:AlazardBurqZuilyExw/ST}.

Our main result will be derived in three different flavors, which we believe are all equally interesting.
The first one controls the dynamic using only H\"{o}lder norms of the quantities.
\begin{thm} \label{thm:mainhold}
	Let $d\geq1$, $s>1+d/2$, $\eps>0$ and consider $(\eta_0,\psi_0)$ satisfying the assumptions 
	of Theorem~\ref{thm:AlazardBurqZuilyExw/ST}.
	If~$T$ is the maximum existence time of the solution given by this theorem, then either~$T=+\infty$ or one of the following quantities is infinite
		\begin{itemize}
			\item $\sup_{0\leq t< T}\frac{1}{h(t)},$
			\item $\sup_{0\leq t< T}\frac{1}{c(t)},$
			\item $\sup_{0\leq t< T}\lA\eta(t)\rA_{W^{1+\eps,\infty}(\R^d)},$
			\item $\sup_{0\leq t< T}\lA(V,B)(t)\rA_{W^{\eps,\infty}(\R^d)},$
			\item $\sup_{0\leq t< T}\lA a(t)\rA_{W^{\eps,\infty}(\R^d)},$
			\item $\int_0^T\lA(\partial_t a+V\cdot\nabla a)(t)\rA_{L^\infty(\R^d)}\d t,$
			\item $\int_0^T\lA a(t)\rA_{W^{\frac{1}{2},\infty}(\R^d)}\d t,$
			\item $\int_0^T\lA\nabla\eta(t)\rA^3_{W^{\frac{1}{2},\infty}(\R^d)}\d t,$
			\item $\int_0^T\lA(V,B)(t)\rA^3_{W^{1+\eps,\infty}(\R^d)}\d t.$
		\end{itemize}
	Here~$h(t)$ is the largest~$h$ satisfying condition~(\ref{eq:BotCnd}) at time~$t$ and~$c(t)$ the largest~$c$ such that~$a(t,x)\geq c$ for all~$x\in\R^d$.
\end{thm}

Before introducing the second criterion, we observe that in the case where the domain is infinitely deep (that is~$\Gamma=\emptyset$), the equation enjoys a scaling invariance. The critical space corresponds to the index~$s=d/2+1/2$.
We expect to find a better criterion by authorizing a control of a Sobolev norm of a fixed reference index~$s_0$ close to the scaling. This corresponds to our second result
\begin{thm} \label{thm:mainsob}
	Let $d\geq1$, $s>1+d/2$, $s>s_0>1/2+d/2$ and $s_0-1/2-d/2>\eps>0$, and consider $(\eta_0,\psi_0)$ satisfying the assumptions 
	of Theorem~\ref{thm:AlazardBurqZuilyExw/ST}.
	If~$T$ is the maximum existence time of the solution given by this theorem, then either~$T=+\infty$ or one of the following quantities is infinite
		\begin{itemize}
			\item $\sup_{0\leq t< T}\frac{1}{h(t)},$
			\item $\sup_{0\leq t< T}\frac{1}{c(t)},$
			\item $\sup_{0\leq t< T}\lA(\eta,\psi,V,B)(t)\rA_{H^{s_0+\mez}(\R^d)\times H^{s_0+\mez}(\R^d)\times H^{s_0}(\R^d)\times H^{s_0}(\R^d)},$
			\item $\sup_{0\leq t< T}\lA a(t)\rA_{W^{\eps,\infty}(\R^d)},$
			\item $\int_0^T\lA(\partial_t a+V\cdot\nabla a)(t)\rA_{L^\infty(\R^d)}\d t,$
			\item $\int_0^T\lA a(t)\rA_{W^{\frac{1}{2},\infty}(\R^d)}\d t,$
			\item $\int_0^T\lA\nabla\eta(t)\rA_{W^{\frac{1}{2},\infty}(\R^d)}\d t,$
			\item $\int_0^T\lA(V,B)(t)\rA_{W^{1+\eps,\infty}(\R^d)}\d t.$
		\end{itemize}
	Here~$h(t)$ is the largest~$h$ satisfying condition~(\ref{eq:BotCnd}) at time~$t$ and~$c(t)$ the largest~$c$ such that~$a(t,x)\geq c$ for all~$x\in\R^d$.
\end{thm}
Here the main improvement to the preceding theorem is that we only need to control the $L^1$-norm in time of the higher order quantities, rather than~$L^3$ norms.
The proofs of those two theorems will be parallel, however one can not be deduced from the other.

The last criterion is a simplification of the preceding one, and is the most compact of the three. It trades a higher reference Sobolev index~$s_0>3/4+d/2$ against control of the Taylor coefficient.
\begin{thm} \label{thm:mainsob2}
	Let $d\geq1$, $s>1+d/2$, $s>s_0>3/4+d/2$ and $1/4>\eps>0$, and consider $(\eta_0,\psi_0)$ satisfying the assumptions of Theorem~\ref{thm:AlazardBurqZuilyExw/ST}.
	If~$T$ is the maximum existence time of the solution given by this theorem, then either~$T=+\infty$ or one of the following quantities is infinite
		\begin{itemize}
			\item $\sup_{0\leq t< T}\frac{1}{h(t)},$
			\item $\sup_{0\leq t< T}\frac{1}{c(t)},$
			\item $\sup_{0\leq t< T}\lA(\eta,\psi,V,B)(t)\rA_{H^{s_0+\mez}(\R^d)\times H^{s_0+\mez}(\R^d)\times H^{s_0}(\R^d)\times H^{s_0}(\R^d)},$
			\item $\int_0^T\lA\nabla\eta(t)\rA_{W^{\frac{1}{2},\infty}(\R^d)}\d t,$
			\item $\int_0^T\lA(V,B)(t)\rA_{W^{1+\eps,\infty}(\R^d)}\d t.$
		\end{itemize}
	Here~$h(t)$ is the largest~$h$ satisfying condition~(\ref{eq:BotCnd}) at time~$t$ and~$c(t)$ the largest~$c$ such that~$a(t,x)\geq c$ for all~$x\in\R^d$.
\end{thm}
Let us explain why the index $3/4+d/2$ enters into the analysis. As already mentioned, in the recent paper~\cite{AlazardBurqZuilyStrEstGrav}, Alazard, Burq and Zuily used Strichartz estimates to deduce existence for data with regularity associated to~$s=11/12+d/2$. The theoretical limit 
of this method is at~$s>3/4+d/2$ for $d=2$ (since Strichartz estimate gains only $1/4$ derivative), and even for such solutions we expect the quantities in this last theorem to be finite on the existence time interval. This would not be the case of the 
quantities~$\sup_{0\leq t< T}\lA\eta(t)\rA_{W^{2,\infty}(\R^d)}$, corresponding to the curvature of the surface, or~$\int_0^T\lA(V,B)(t)\rA^3_{W^{1+\eps,\infty}(\R^d)}\d t$, since 
solutions can be found for which those quantities would be infinite.

Section~\ref{sec:EllReg} will start with a rigorous definition of the harmonic
extension~$\phi$ and of the Dirichlet-Neumann, adapted to the case with rough bottom. It also contains in subsection~\ref{subsec:MaxPrinc} a maximum principle adapted to this framework, that we believe is of
independent interest and ends with 
results on the elliptic regularity of this problem and their uses to control the Dirichlet-Neumann. In section~\ref{sec:Paralin} we will perform the reduction of the system to a symmetric quasilinear 
hyperbolic equation. This imposes to change the variables we work with; in section~\ref{sec:OrUnk} we will construct a parametrix to control the new variables with the originals. Section~\ref{sec:EneEst}
contains the a priori energy estimates of the new system, and section~\ref{sec:Proof} completes the proofs of the theorems.
Appendix~\ref{ap:para} recalls some notions on paradifferential calculus, the main technical tool of this analysis.

\section{Elliptic Regularity} \label{sec:EllReg}

Following the general strategy of \cite{AlazardBurqZuilyExw/ST}, the first step of the proof is to estimate solutions of the Laplace equation near the free surface. 
The method is essentially the same, but we look for tame estimates whose constants depend on the H\"{o}lder norm of the surface rather 
than on its Sobolev norm. This requires some new techniques, and in particular we shall prove a maximum principle adapted to this setting.
This analysis being valid at fixed time, we will drop the dependence in $t$ for this whole section.

\subsection{Variational solution} \label{subsec:vardef}

We have to give a suitable sense to quantities defined in $\Omega$, from data defined only on the free surface.
Here, we recall this construction.\\
Those quantities need to be, in a suitable sense, solutions of
\begin{equation} \label{eq:lapeq}
\Delta_{x,y}v=0,\quad v|_\Sigma=f,\quad\partial_nv|_\Gamma=0.
\end{equation}
This definition will come from variational theory.

\begin{nota}
Let $\mathcal{D}$ be the space of functions $u\in C^\infty(\Omega)$ such that $\nabla_{\!x,y}u\in L^2(\Omega)$.\\
Let then $\mathcal{D}_0$ be the subspace of $\mathcal{D}$ whose elements are equal to 0 near the top boundary $\Sigma$.
\end{nota}

\begin{prop}(\cite[Proposition~2.2]{AlazardBurqZuilyExwST}) \label{prop:weight}
There exists a positive weight $g\in L^{\infty}_{\text{loc}}(\Omega)$,
equal to 1 near the top boundary of $\Omega$, and a constant $C>0$ such that for all $u\in\mathcal{D}_0$,
$$\int_\Omega g(x,y)\lvert u(x,y)\rvert^2\;\mathrm{d}x\;\mathrm{d}y\leq C\int_{\Omega}\lvert\nabla_{\!x,y}u(x,y)\rvert^2\;\mathrm{d}x\;\mathrm{d}y.$$
\end{prop}

\begin{defn}
Let $\mathcal{H}^{1,0}$ be the space of functions $u$ on $\Omega$ such that there exists a sequence~$(u_n)_{n\in\N}$, with~$u_n\in D_0$, satisfying
$$\nabla_{\!x,y}u_n\rightarrow\nabla_{\!x,y}u\text{ in }L^2(\Omega,\d x\d y),
\quad u_n\rightarrow u\text{ in }L^2(\Omega,g(x,y)\d x\d y).$$
\end{defn}
We see from Proposition~\ref{prop:weight} that $\mathcal{H}^{1,0}$ can be equipped with the norm
$$\lA u\rA_{\mathcal{H}^{1,0}}=\lA \nabla_{\!x,y}u\rA_{L^2(\Omega)}.$$\\
As seen in \cite{AlazardBurqZuilyExwST}, it is a Hilbert space.
By regularizing the function~$\eta$, we can construct~$\eta_*\in C^\infty_b(\R^d)$ such that 
$\eta-h/20>\eta_*$ and
$$\left\{(x,y)\in\R^d\times\R;\eta_*(x)<y<\eta(x)\right\}\subset\Omega.$$

Recall that~$\mathcal{O}$ denotes the fixed container in which the fluid is located.
\begin{defn}
We denote by~$\mathcal{H}^1(\mathcal{O})$ the space of functions~$\widetilde{u}$ on~$\mathcal{O}$ such that there exists a sequence~$(u_n)\in C^{\infty}(\mathcal{O})$ 
such that
$$\nabla_{\!x,y}u_n\rightarrow\nabla_{\!x,y}\widetilde{u}\text{ in }L^2(\mathcal{O},\d x\d y),\quad u_n\rightarrow\widetilde{u}\text{ in }L^2(\mathcal{O},\hat{g}(x,y)\d x\d y),$$
where~$\hat{g}$ is the extension of~$g$ by 1 to~$\mathcal{O}$.
\end{defn}

\begin{lem} \label{lem:charnulltrace}
Let~$w$ be measurable on~$\Omega$.
Then~$w\in\mathcal{H}^{1,0}(\Omega)$ if and only if the zero extension of~$w$ to~$\mathcal{O}$ is in~$\mathcal{H}^1(\mathcal{O})$. 
\end{lem}
\begin{proof}
We follow the proof for the classical Sobolev setting, found for example in~\cite{AdamsSobSpaces}.
It is routine to show that $\widetilde{\nabla_{\!x,y}w}=\nabla_{\!x,y}\widetilde{w}$, from which the direct part is immediate.

For the indirect part, suppose~$\widetilde{w}\in\mathcal{H}^1(\mathcal{O})$.
Now we can cover~$\Omega$ with $V_0$ which does not intersect $\Sigma$ and $V_1$ which does not intersect $\Gamma$.
Then using a partition of unity, we can split $w$ between $w_0$ supported in~$V_0$, which by definition is already in $\mathcal{H}^{1,0}(\Omega)$,
and $w_1$ supported in $V_1$. Then we consider $\widetilde{w_1}(x, y+t)$, which is in $\mathcal{H}^{1,0}(\Omega)$ and converge to $w_1$
as $t$ goes to $0^+$, since the translation in $L^2$ is continuous. This proves that $w_2$, and then $w$ is in $\mathcal{H}^{1,0}(\Omega)$.
\end{proof}

Let $f\in H^{1/2}(\R^d)$. We define $\psi$ an $H^1$ lifting of $f$ in $\Omega$.
Let $\chi_0(z)\in C^\infty(R)$ be such that $\chi_0(z)=1$ if $z\geq-1/2$ and $\chi_0(z)=0$ if $z\leq-1$. Set
$$\psi_1(x,z):=\chi_0(z)e^{z\lvert D_x\rvert}f(x),\;x\in\R^d,z\leq0.$$
Then set 
$$\psi(x,y):=\psi_1\left(x,\frac{y-\eta(x)}{h}\right),\;(x,y)\in\Omega,$$
which is well defined and vanishes near the bottom $\Gamma$.

>From the usual properties of the Poisson kernel, we have

\begin{equation*}
\lA\psi\rA_{H^1(\Omega)}\leq
\mathcal{F}\left(\lA\eta\rA_{W^{1,\infty}(\R^d)}\right)\lA f\rA_{H^{1/2}(\R^d)},
\end{equation*}
and

\begin{equation*}
\lA\psi\rA_{L^\infty(\Omega)}\leq
\lA f\rA_{L^{\infty}(\R^d)}.
\end{equation*}

We can now use this framework to define $u\in\mathcal{H}^{1,0}$ as a variational solution of the problem

\begin{equation*}
-\Delta_{x,y}u=\Delta_{x,y}\psi,\quad u|_\Sigma=0,\quad\partial_nu|_\Gamma=0.
\end{equation*}

We then define 
\begin{equation}
v:=u+\psi.\label{eq:defexthar}
\end{equation}

We see from lemma~3.5 of \cite{AlazardBurqZuilyExw/ST} that this is independent of the lifting function $\psi$ vanishing near the bottom and we freely get the estimate

\begin{equation*}
  \int_\Omega\lvert\nabla_{\!x,y}v\rvert^{2}\d x\d y\leq
  \mathcal{F}\left(\lA\eta\rA_{W^{1,\infty}(\R^d)}\right)\lA f\rA^2_{H^{1/2}(\R^d)}.
\end{equation*}

\subsection{Maximum principle} \label{subsec:MaxPrinc}
In studying equation~(\ref{eq:lapeq}) we will need a weak maximum principle adapted to our variational formulation.
Adapting the proof from \cite{TrudingerMaxprinc}, we get the following comparison principle.

\begin{prop} \label{prop:compprinc}
If $\phi$ is a weakly differentiable function such that:
\begin{enumerate}
\item $\phi^+=\max(\phi,0)\in\mathcal{H}^{1,0}$;\label{cndition1}
\item $\int_\Omega\nabla_{\!x,y}\phi\cdot\nabla_{\!x,y}\sigma\d x\d y\leq0$ for all $\sigma\geq0$ in $\mathcal{H}^{1,0}$;\label{cnditon2}
\end{enumerate}
then $\phi\leq0$ in $\Omega$.
\end{prop}
\begin{rem}
Condition \ref{cndition1} is the adapted way to say that $\phi|_\Sigma\leq0$ for the variational space $\mathcal{H}^{1,0}$.
\end{rem}
\begin{proof}
Since $\phi^+\geq0$, and $\phi^+\in\mathcal{H}^{1,0}$, we have from condition \ref{cnditon2}, taking $\sigma=\phi^+$ 
$$\lA \phi^+\rA^2_{\mathcal{H}^{1,0}}=\int_\Omega\nabla_{\!x,y}\phi^+\cdot\nabla_{\!x,y}\phi^+\d x\d y\leq0$$
so that $\phi^+=0$, which is the desired conclusion.
\end{proof}

We can now extend this comparison principle to get the following maximum principle.
\begin{prop} \label{prop:maxprinc}
Let $\eta\in W^{1,\infty}(\R^d)$, $f\in H^{1/2}(\R^d)$.
If $v$ is the solution of Laplace equation defined in (\ref{eq:defexthar}), and if $f$ is bounded, then
$$\lA v\rA_{L^{\infty}(\Omega)}\leq\lA f\rA_{L^{\infty}(\R^d)}.$$
\end{prop}

\begin{proof}
Keeping in mind the preceding theorem, the only thing we need to prove is that 
$(v-(1+\eps)\lA f\rA_{L^{\infty}(\R^d)})^+\in\mathcal{H}^{1,0}$.
Replacing $v$ with $-v$ and letting $\eps$ go to $0^+$ will then complete the proof.

To prove this claim, we will use Lemma~\ref{lem:charnulltrace}.
Since~$v\in\mathcal{H}^{1,0}(\Omega)$, the zero extension~$\widetilde{v}$ is in~$\mathcal{H}^1(\mathcal{O})$.
As in subsection~\ref{subsec:vardef}, we can extend~$f$ to~$\mathcal{O}$ using the Poisson kernel $e^{\eps z\langle D_x\rangle}$.
This extension~$\widetilde{f}$ is bounded by~$(1+\eps)\lA f\rA_{L^\infty}$ (see Lemma~\ref{lem:poissker}), so that~$\left(\widetilde{v}+\widetilde{f}-(1+\eps)\lA f\rA_{L^\infty}\right)^+$
is in~$\mathcal{H}^1(\mathcal{O})$ by elementary properties of this space, is zero on~$\mathcal{O}\setminus\Gamma$, and so by Lemma~\ref{lem:charnulltrace}
$\left(\widetilde{v}+\widetilde{f}-(1+\eps)\lA f\rA_{L^\infty}\right)^+\in\mathcal{H}^{1,0}(\Omega)$.
\end{proof}

We will mainly use the following classical consequence of the maximum principle:
\begin{prop} \label{prop:maxprincder}
  If $0<h'<h$, and $\Omega_{h'}=\{(x,y)\in\R^d\times\R,\eta(x)-h'<y<\eta(x)\}$, there exists a constant $C_{h'}>0$ such that 
  if $f\in C^{1+\eps}(\R^d)\cap H^\mez(\R^d)$ and $v$ is a variational solution of (\ref{eq:lapeq}),
  $$\lA v\rA_{C^{1+\eps}(\Omega_{h'})}\leq C_{h'}\lA f\rA_{C^{1+\eps}(\R^d)}.$$
\end{prop}
\begin{proof}
Noticing that $v$ is an $H^1$ variational solution of $\Delta_{x,y}v=0$ in $\Omega_h$, Corollary~8.36 of \cite{GilbargTrudingerEllPDE} gives this on compact sub-domains of $\Omega_{h'}$, 
and the constant being uniform, we can deduce the result on the full $\Omega_{h'}$.
\end{proof}

\subsection{Straightening the boundary}

In order to study further regularity of those solutions, it is convenient to straighten the domain, transforming an equation with constant coefficients on a variable domain to 
an equation with variable coefficients on a fixed domain.
As seen in \cite{AlazardBurqZuilyExw/ST}, there exists a function $\eta_*$ such that
\begin{equation}
\left\lbrace
\begin{gathered}
  \eta_*+\frac{h}{4}\in H^\infty(\R^d),\\
  \eta(x)-\eta_*(x)=\frac{h}{4}+g,\quad\lA g\rA_{L^\infty(\R^d)}\leq\frac{h}{5},\\
  \Gamma\subset\lbrace(x,y)\in\mathcal{O}:y<\eta_*(x)\rbrace.
\end{gathered}
\right.
\end{equation}
We can take for example
$$\eta_*(x)=-\frac{h}{4}+e^{-\nu\langle D_x\rangle}\eta(x),$$
where $\nu>0$ is such that $\nu\lA\eta\rA_{W^{1,\infty}(\R^d)}\leq\frac{h}{5}$.

This gives 
$$\lA g\rA_{L^\infty(\R^d)}=\lA e^{-\nu\langle D_x\rangle}\eta-\eta\rA_{L^\infty(\R^d)}
  \leq\nu\lA\eta\rA_{W^{1,\infty}(\R^d)}\leq\frac{h}{5},$$
thanks to the following classical lemma.
\begin{lem} \label{lem:poissker}
  Let $f\in W^{1,\infty}(\R^d)$, $c>0$, and $t>0$. Then $e^{-bt\langle D_x\rangle}f\in L^\infty(\R^d)$ and
  $$\lA e^{-bt\langle D _x\rangle}f-f\rA_{L^\infty(\R^d)}\leq C\times bt\lA f\rA_{W^{1,\infty}(\R^d)},$$
  with $C>0$ a constant.
\end{lem}

Now define
\begin{equation*}
\left\lbrace
\begin{aligned}
  \Omega_1&:=\lbrace(x,y:x\in\R^d,\eta_*(x)<y<\eta(x)\rbrace,\\
  \Omega_2&:=\lbrace(x,y)\in\mathcal{O}:y\leq\eta_*(x)\rbrace,\\
  \Omega&:=\Omega_1\cup\Omega_2,
\end{aligned}
\right.
\end{equation*}
and
\begin{equation*}
\left\lbrace
\begin{aligned}
  \widetilde{\Omega}_1&:=\lbrace(x,z):x\in\R^d,z\in I\rbrace,\quad I=(-1,0),\\
  \widetilde{\Omega}_2&:=\lbrace(x,z)\in\R^d\times(-\infty,-1]:(x,z+1+\eta_*(x))\in\Omega_2\rbrace,\\
  \widetilde{\Omega}&:=\widetilde{\Omega}_1\cup\widetilde{\Omega}_2.
\end{aligned}    
\right.
\end{equation*}

Following Lannes (\cite{LannesWellPos}), we consider the map $\rho$ from $\widetilde{\Omega}$ to $\R^d$ defined as
\begin{equation}
\left\lbrace
\begin{aligned} \label{eq:defvar}
  \rho(x,z):=&(1+z)e^{\delta z\langle D_x\rangle}\eta(x)-z\eta_*(x)\quad\text{if }(x,z)\in\widetilde{\Omega}_1,\\
  \rho(x,z):=&z+1+\eta_*(x)\quad\text{if }(x,z)\in\widetilde{\Omega}_2,
\end{aligned}
\right.
\end{equation}
with $\delta=\delta(\lA\eta\rA_{W^{1,\infty}(\R^d)})>0$ small.

Using lemma \ref{lem:poissker}, we have
$$\lA\partial_z\rho-\frac{h}{4}\rA_{L^\infty(\widetilde{\Omega}_1)}\leq\delta C\lA\eta\rA_{W^{1,\infty}(\R^d)}+\frac{h}{5},$$
which, taking $\delta$ small enough, gives
\begin{equation*}
\left\lbrace
\begin{aligned}
  \partial_z\rho(x,z)&\geq\min\left(1,\frac{h}{5}\right),\quad\forall(x,z)\in\widetilde{\Omega},\\
  \lA\nabla_{\!x,z}\rho\rA_{L^\infty(\widetilde{\Omega})}&\leq\mathcal{F}\left(\lA\eta\rA_{W^{1,\infty}(\R^d)}\right).
\end{aligned}
\right.
\end{equation*}
This proves that the map $(x,z)\mapsto(x,\rho(x,z))$ is a $C^1$-diffeomorphism from~$\widetilde{\Omega}$ to~$\Omega$.

\begin{lem}
Let $I=[-1,0]$. We have in Sobolev-type spaces, for any~$\sigma> 1/2+d/2$,
\begin{equation} \label{eq:rhosob}
\begin{gathered}
\lA\partial_z\rho-\frac{h}{4}\rA_{X^{\sigma-\frac{1}{2}}(I)}+\lA\nabla_{\!x}\rho\rA_{X^{\sigma-\frac{1}{2}}(I)}
      \leq\mathcal{F}\left(\lA\eta\rA_{W^{1,\infty}(\R^d)}\right)\lA\eta\rA_{H^{\sigma+\frac{1}{2}}(\R^d)},\\
\lA\partial_z\rho-\frac{h}{4}\rA_{L^1(I;H^{\sigma-\frac{1}{2}}(\R^d))}
	  +\lA\nabla_{\!x}\rho\rA_{L^1(I;H^{\sigma-\frac{1}{2}}(\R^d))}
	  \leq C\lA\eta\rA_{H^{\sigma+\frac{1}{2}}(\R^d)}.
\end{gathered}
\end{equation}
With $0\leq r\leq1/2$, we have in H\"{o}lder-type spaces
\begin{equation} \label{eq:rhohold}
\begin{gathered}
\lA\partial_z\rho-\frac{h}{4}\rA_{C^0(I;C^r_*(\R^d))}
	  +\lA\nabla_{\!x}\rho\rA_{C^0(I;C^r_*(\R^d))}
      \leq\mathcal{F}\left(\lA\eta\rA_{W^{1,\infty}(\R^d)}\right)\lA\eta\rA_{C^{1+r}_*(\R^d)},\\
\lA\partial_z\rho-\frac{h}{4}\rA_{\tilde{L}^2(I;C^{\mez+r}_*(\R^d))}
	  +\lA\nabla_{\!x}\rho\rA_{\tilde{L}^2(I;C^{\mez+r}_*(\R^d))}
	  \leq C[\lA\eta\rA_{L^2(\R^d)}+\lA\eta\rA_{C^{1+r}_*(\R^d)}],\\
\lA\nabla_{x,z}^2\rho\rA_{\tilde{L}^2(I;C^{-\mez+r}_*(\R^d))}
	  \leq C[\lA\eta\rA_{L^2(\R^d)}+\lA\eta\rA_{C^{1+r}_*(\R^d)}].
\end{gathered}
\end{equation}
\end{lem}
\begin{proof}
(\ref{eq:rhosob}) is proved in \cite{AlazardBurqZuilyExw/ST}, and the first part of (\ref{eq:rhohold}) is a straightforward consequence of Lemma \ref{lem:poissker}.
The second part of (\ref{eq:rhohold}) will be proved with Littlewood-Paley decomposition, following \cite{WangZhangBrkWW}.
We have $\lA\nabla_{\!x}\rho\rA_{\tilde{L}^2(I;C^{\mez+r}_*(\R^d))}
  \lesssim\lA\langle D_x\rangle\rho\rA_{\tilde{L}^2(I;C^{\mez+r}_*(\R^d))},$ so that we only need to control this last norm.
We can split $\langle D_x\rangle\rho$ in two parts, 
$(1+z)e^{\delta z\langle D_x\rangle}\langle D_x\rangle\eta$ and $-ze^{-\nu\langle D_x\rangle}\langle D_x\rangle\eta$, whose norms will be respectively $A$ and $B$.
For the first part, we have
\begin{align*}
  A&\leq\sup_{j>0}2^{j(\mez+r)}\lA\Delta_je^{\delta z\langle D_x\rangle}\langle D_x\rangle\eta\rA_{L^2(I;L^\infty(\R^d))}
    +\lA\Delta_0e^{\delta z\langle D_x\rangle}\langle D_x\rangle\eta\rA_{L^1(I;L^\infty(\R^d))}\\
    &\leq C\sup_{j>0}2^{j(\frac{3}{2}+r)}\lA e^{c\delta z2^j}\rA_{L^{2}(I)}\lA\Delta_j\eta\rA_{L^{\infty}(\R^d)}
      +\lA e^{\delta z\langle D_x\rangle}\Delta_0\eta\rA_{L^2(I\times\R^d)}\\
    &\leq C\sup_{j\geq0}2^{j(1+r)}\lA\Delta_j\eta\rA_{L^\infty(\R^d)}+\lA\Delta_0\eta\rA_{L^{2}(\R^d)},
\end{align*}
where we have used extensively Bernstein estimates (\ref{lem:Berns}) and the smoothing effect (\ref{lem:parabsmooth}).
The same method applies for $B$, taking $z=-1$ and bounding $e^{-c\nu2^j}$ by $1/2^j$.
This gives the expected result, and computations for $\partial_z\rho-h/4$ and the second-order terms are identical.
\end{proof}

Now for a function $v$ defined on $\Omega$, put
$$\widetilde{v}(x,z):=v(x,\rho(x,z)).$$
We then have 
\begin{equation}
\left\lbrace
\begin{gathered} \label{eq:Lambda}
  (\partial_yv)(x,\rho(x,z))=\Lambda_1\widetilde{v}(x,z),\quad(\nabla_{\!x}v)(x,\rho(x,z))=\Lambda_2\widetilde{v}(x,z),\\
  \Lambda_1:=\frac{1}{\partial_z\rho}\partial_z,\quad\Lambda_2:=\nabla_{\!x}-\frac{\nabla_{\!x}\rho}{\partial_z\rho}\partial_z.
\end{gathered}
\right.
\end{equation}

If $v$ is a solution of $\Delta_{x,y}v=F_0$ in $\Omega$ then $\widetilde{v}$ satisfies
$$(\Lambda_1^2+\Lambda_2^2)\widetilde{v}=\widetilde{F}_0\quad\text{in }\widetilde{\Omega}.$$
This equation can be expanded to
\begin{equation} \label{eq:lapeqflat}
\left\lbrace
\begin{gathered}
  (\partial_z^2+\alpha\Delta_{x}+\beta\cdot\nabla_{\!x}\partial_z-\gamma\partial_z)\widetilde{v}=\widetilde{F}_0,\\
  \widetilde{v}(z=0)=f,\\
  \alpha:=\frac{(\partial_z\rho)^2}{1+\lvert\nabla_{\!x}\rho\rvert^2},\quad\beta:=-2\frac{\partial_z\rho\nabla_{\!x}\rho}{1+\lvert\nabla_{\!x}\rho\rvert^2},\quad
    \gamma:=\frac{1}{\partial_z\rho}(\partial_z^2\rho+\alpha\Delta_{x}\rho+\beta\cdot\nabla_{\!x}\partial_z\rho).
\end{gathered}
\right.
\end{equation}

We also remark that 
$$(\Lambda_1^2+\Lambda_2^2)\rho=0\quad\text{in }\widetilde{\Omega},$$
and that 
$$[\Lambda_1,\Lambda_2]=0.$$

We now derive some estimates on the new coefficients.
\begin{lem} \label{lem:ellcoefs}
Let $I=[-1,0]$. We have for~$\sigma> 1/2+d/2$
\begin{align}
\lA\alpha-\frac{h^2}{16}\rA_{X^{\sigma-\frac{1}{2}}(I)}
      &\leq\mathcal{F}\lp\lA\eta\rA_{W^{1,\infty}(\R^d)}\rp\lA\eta\rA_{H^{\sigma+\frac{1}{2}}(\R^d)}, \label{eq:alphasob}\\
\lA\beta\rA_{X^{\sigma-\frac{1}{2}}(I)}
      &\leq\mathcal{F}\lp\lA\eta\rA_{W^{1,\infty}(\R^d)}\rp\lA\eta\rA_{H^{\sigma+\frac{1}{2}}(\R^d)}, \label{eq:betasob}\\
\lA\gamma\rA_{X^{\sigma-\frac{3}{2}}(I)}&\leq \mathcal{F}\lp\lA\eta\rA_{W^{1,\infty}(\R^d)}\rp 
    \lb1+\lA\eta\rA_{H^{\sigma+\frac{1}{2}}(\R^d)}\rb, \label{eq:gammasob}
\end{align}
and for any $0\leq r\leq1/2$, we have
\begin{align}
\lA\alpha-\frac{h^2}{16}\rA_{C^{0}(I;C^{r}_*(\R^d))}
      &\leq\mathcal{F}\lp\lA\eta\rA_{W^{1,\infty}(\R^d)}\rp\lA\eta\rA_{C^{1+r}_*(\R^d)}, \label{eq:alphahold}\\
\lA\beta\rA_{C^{0}(I;C^{r}_*(\R^d))}
      &\leq\mathcal{F}\lp\lA\eta\rA_{W^{1,\infty}(\R^d)}\rp\lA\eta\rA_{C^{1+r}_*(\R^d)}. \label{eq:betahold}
\end{align}

\end{lem}
\begin{proof}
We see from (\ref{eq:rhosob}) that we can write 
$$(\partial_z\rho)^2=\frac{h^2}{16}+G\text{ with }\lA G\rA_{X^{\sigma-\frac{1}{2}}(I)}
    \leq\mathcal{F}(\lA\eta\rA_{W^{1,\infty}(\R^d)})\lA\eta\rA_{H^{\sigma+\frac{1}{2}}(\R^d)}.$$
We can now decompose 
$$\alpha-\frac{h^2}{16}=-\frac{h^2}{16}\times\frac{\lvert\nabla_{\!x}\rho\rvert^2}{1+\lvert\nabla_{\!x}\rho\rvert^2}
      +G-G\frac{\lvert\nabla_{\!x}\rho\rvert^2}{1+\lvert\nabla_{\!x}\rho\rvert^2},$$
and we use the tame estimates of (\ref{eq:tamestsob}) to conclude. 
The other inequalities are all proved with the same method, using the other estimates of (\ref{eq:rhosob}) and (\ref{eq:rhohold}).
\end{proof}

\subsection{Elliptic regularity in the new domain}

We now study the new equation (\ref{eq:lapeqflat}), following the method from \cite{AlazardBurqZuilyExw/ST}, with tame estimates at every step.
Recall from~(\ref{eq:parabspaces}) the definition of the spaces
\begin{equation*}
\begin{aligned}
  X^\mu(I)&=L^\infty_z(I;H^\mu(\R^d))\cap L^2_z(I;H^{\mu+\frac{1}{2}}(\R^d)),\\
  Y^\mu(I)&=L^1_z(I;H^\mu(\R^d))+L^2_z(I;H^{\mu-\frac{1}{2}}(\R^d)).
\end{aligned}
\end{equation*}

We suppose $\widetilde{v}$ to be a solution of (\ref{eq:lapeqflat}) in $I\times\R^d$, $I=[-1,0]$ with the additional assumption
\begin{equation} \label{eq:assumpsob}
\lA\widetilde{v}\rA_{X^{-\frac{1}{2}}([-1,0])}<\infty.
\end{equation}
We know from \cite{AlazardBurqZuilyExw/ST} that this estimation holds for our variational solutions of~(\ref{eq:lapeq}), with
$$\lA\widetilde{v}\rA_{X^{-\frac{1}{2}}([-1,0])}\leq\mathcal{F}\lp\lA\eta\rA_{W^{1,\infty}(\R^d)}\rp\lA f\rA_{H^{\frac{1}{2}}(\R^d)}.$$

The main result of the section will be stated in two versions, corresponding to our two main theorems.
\begin{prop} \label{prop:regsob}
Let~$\epsilon>0$, $s_0>1/2+d/2$, and $s>1+d/2$. Let
$$-\frac{1}{2}\leq\sigma\leq s-\frac{1}{2}.$$
Consider $f\in H^{\sigma+1}(\R^d)$, $\widetilde{F}\in Y^{\sigma}(I)$ and $\widetilde{v}$ satisfying the hypothesis (\ref{eq:assumpsob}), solution to (\ref{eq:lapeqflat}).
Then for any $z_0\in(-1,0)$,
\begin{multline} \label{eq:elliphold}
\lA\nabla_{\!x,z}\widetilde{v}\rA_{X^{\sigma}([z_0,0])}\leq\mathcal{F}\lp\lA\eta\rA_{(C^{1+\eps}_*\cap L^2)(\R^d)}\rp\left[\lA\nabla_{\!x}f\rA_{H^{\sigma}(\R^d)}
	+\lA \widetilde{F}_0\rA_{Y^{\sigma}(I)}\right.\\
    \left.+\lA\eta\rA_{H^{s+\frac{1}{2}}(\R^d)}\lA\nabla_{\!x,z}\widetilde{v}\rA_{L^\infty(I\times\R^d)}
	+\lA\nabla_{\!x,z}\widetilde{v}\rA_{X^{-\frac{1}{2}}([-1,0])}\right], 
\end{multline}
and
\begin{equation} \label{eq:ellipsob}
\lA\nabla_{\!x,z}\widetilde{v}\rA_{X^{\sigma}([z_0,0])}\leq\mathcal{F}\lp\lA\eta\rA_{H^{s_0+\mez}(\R^d)}\rp\left[\lA\nabla_{\!x}f\rA_{H^{\sigma}(\R^d)}
	+\lA \widetilde{F}_0\rA_{Y^{\sigma}(I)}
    +\lA\nabla_{\!x,z}\widetilde{v}\rA_{X^{-\frac{1}{2}}([-1,0])}\right]. 
\end{equation}
\end{prop}
\begin{rem}
The proof will show that the function $\mathcal{F}$ depends on $z_0$ and on $\sigma$. On every application, those parameters will be fixed, and independent of time.
\end{rem}

As in \cite{AlazardBurqZuilyExw/ST}, this will be proved by induction on the regularity $\sigma$.
The property that inequalities (\ref{eq:elliphold}) and~(\ref{eq:ellipsob})  hold for $\sigma$, for all admissible $z_0$, will be denoted by~$\mathcal{H}_\sigma$.
In this notation, hypothesis~(\ref{eq:assumpsob}) means that $\mathcal{H}_{-1/2}$ is satisfied.

Proposition \ref{prop:regsob} is then a consequence of
\begin{prop} \label{prop:regsobboot}
For any $\delta$ such that 
  $$0<\delta\leq\inf\lp\eps,s_0-\mez-\frac{d}{2}\rp\leq\frac{1}{2},$$
if $\mathcal{H}_\sigma$ is satisfied for some $-1/2\leq\sigma\leq s-1/2-\delta$, then $\mathcal{H}_{\sigma+\delta}$ is satisfied.
\end{prop}

To prove this, we first estimate the lower order term.
\begin{lem} \label{lem:estrellsob1}
For all $J\subset I$,
\begin{equation}
  \lA\widetilde{F}_1\rA_{Y^{\sigma+\delta}(J)}\leq\mathcal{F}\lp\lA\eta\rA_{W^{1,\infty}\cap L^2}\rp\lp1+\lA\eta\rA_{C^{1+\delta}_*}\rp
    \left[\lA\eta\rA_{H^{s+\frac{1}{2}}}\lA\partial_z\widetilde{v}\rA_{L^\infty(\R^d\times J)}+\lA\partial_z\widetilde{v}\rA_{X^{\sigma}(J)}\right],
\end{equation}
and
\begin{equation}
  \lA\widetilde{F}_1\rA_{Y^{\sigma+\delta}(J)}\leq\mathcal{F}\lp\lA\eta\rA_{H^{s_0+\mez}}\rp\lA\partial_z\widetilde{v}\rA_{X^{\sigma}(J)},
\end{equation}
with $\widetilde{F}_1=\gamma\partial_z\widetilde{v}$.
\end{lem}
\begin{proof}
For the first inequality, we decompose~$\widetilde{F}_1$ as 
$$\widetilde{F}_1=\partial_z^2\rho\frac{1}{\partial_z\rho}\partial_z\widetilde{v}+\Delta_x\rho\frac{\alpha}{\partial_z\rho}\partial_z\widetilde{v}
+\nabla_{\!x}\partial_z\rho\cdot\frac{\beta}{\partial_z\rho}\partial_z\widetilde{v}:=A+B+C.$$
We then use the paraproduct estimates~(\ref{eq:pararem1}), (\ref{eq:paraprod1}), and~(\ref{eq:paraprod2}) to get
\begin{align*}
\lA A\rA_{Y^{\sigma+\delta}}&\leq\lA A\rA_{L^2(J;H^{s+\delta-\mez})}\\
&\leq\lA\partial_z^2\rho\rA_{\widetilde{L}^2\lp J;C^{\delta-\mez}\rp}\lA\frac{1}{\partial_z\rho}\partial_z\widetilde{v}\rA_{X^\sigma}
+\lA\frac{1}{\partial_z\rho}\partial_z\widetilde{v}\rA_{L^\infty\lp J\times\R^d\rp}\lA\partial_z^2\rho\rA_{L^2\lp J;H^{s-1}\rp},	
\end{align*}
and using the estimates~(\ref{eq:rhohold}) and tame estimates gives the majoration for~$A$.
Estimates on~$B$ and~$C$ follows along the same lines.
For the second one, paraproduct estimates give immediately 
\begin{equation*}
	\lA\gamma\partial_z\widetilde{v}\rA_{L^1(J;H^{\sigma+\delta})}\leq\lA\partial_z\widetilde{v}\rA_{L^2(J;H^{\sigma+\mez})}\lA\gamma\rA_{L^2(J;H^{s_0-1})},
\end{equation*}
and using the estimate on~$\gamma$ from Lemma~\ref{lem:ellcoefs} enable us to conclude.
\end{proof}

We now replace multiplication by $\alpha$ (resp. $\beta$) with paramultiplication by $T_\alpha$ (resp. $T_\beta$).
\begin{lem} \label{lem:estrellsob2}
$\widetilde{v}$ satisfies the paradifferential equation 
\begin{equation*}
\partial_z^2\widetilde{v}+T_\alpha\Delta_x\widetilde{v}+T_{\beta}\cdot\nabla_{\!x}\partial_z\widetilde{v}=\widetilde{F}_0+\widetilde{F}_1+\widetilde{F}_2,
\end{equation*}
for some remainder 
\begin{equation*}
\widetilde{F}_2=(T_\alpha-\alpha)\Delta_x\widetilde{v}+(T_{\beta}-\beta)\cdot\nabla_{\!x}\partial_z\widetilde{v}
\end{equation*}
satisfying  for all $J\subset I$
\begin{equation*}
\lA\widetilde{F}_2\rA_{Y^{\sigma+\delta}(J)}\leq
      \mathcal{F}(\lA\eta\rA_{W^{1,\infty}(\R^d)})\left[1
      +\lA\eta\rA_{H^{s+\frac{1}{2}}(\R^d)}\right]\lA\nabla_{\!x,z}\widetilde{v}\rA_{L^\infty(\R^d\times J)},
\end{equation*}
or 
\begin{equation*}
	\lA\widetilde{F}_2\rA_{Y^{\sigma+\delta}(J)}\leq
      \mathcal{F}(\lA\eta\rA_{H^{s_0}(\R^d)})\lA\nabla_{\!x,z}\widetilde{v}\rA_{X^\sigma(J)}.
\end{equation*}

\end{lem}
\begin{proof}
First, we have 
$$\widetilde{F}_2=\left(T_{\alpha-\frac{h^2}{16}}-\left(\alpha-\frac{h^2}{16}\right)\right)\Delta_x\widetilde{v}
  -\frac{h^2}{16}(T_1-1)\Delta_x\widetilde{v}+(T_{\beta}-\beta)\cdot\nabla_{\!x}\partial_z\widetilde{v},$$
which, according to~(\ref{eq:paraprod2}) and~(\ref{eq:paradifremsob}) gives
\begin{multline*}
\lA\widetilde{F}_2\rA_{Y^{\sigma+\delta}}\leq C\left[\lA\Delta_x\widetilde{v}\rA_{L^{\infty}C^{-1}_*}
    \left(\lA\alpha-\frac{h^2}{16}\rA_{L^2H^{\sigma+\delta+\frac{1}{2}}}+1\right)\right.\\
    \left.+\lA\nabla_{\!x}\partial_z\widetilde{v}\rA_{L^{\infty}C^{-1}_*}\lA\beta\rA_{L^2H^{\sigma+\delta+\frac{1}{2}}}\right].
\end{multline*}
As $\sigma+\delta+1/2\leq s$, we have
$$\lA\widetilde{F}_2\rA_{Y^{\sigma+\delta}}\leq C\lb1+\lA\alpha-\frac{h^2}{16}\rA_{X^{s-\frac{1}{2}}}+\lA\beta\rA_{X^{s-\frac{1}{2}}}\rb
    \lA\nabla_{\!x,z}\widetilde{v}\rA_{L^\infty}.$$
Using the estimates (\ref{eq:alphasob}) and (\ref{eq:betasob}) completes the proof of the first inequality.

Using again (\ref{eq:paraprod2}) and (\ref{eq:paradifremsob}) gives
\begin{equation*}
	\lA\lp T_{\alpha-\frac{h^2}{16}}-\lp\alpha-\frac{h^2}{16}\rp\rp\Delta_x\widetilde{v}\rA_{L^1H^{\sigma+\delta}}\leq 
	C\lA\Delta_x\widetilde{v}\rA_{L^2H^{\sigma-\mez}}\lp\lA\alpha-\frac{h^2}{16}\rA_{L^2H^{s_0+\mez}}+1\rp,
\end{equation*}
and
\begin{equation*}
	\lA\lp T_\beta-\beta\rp\cdot\nabla_{\!x}\partial_z\widetilde{v}\rA_{L^1H^{\sigma+\delta}}\leq
	C\lA\nabla_{\!x}\partial_z\widetilde{v}\rA_{L^2H^{\sigma-\mez}}\lA\beta\rA_{L^2H^{s_0}},
\end{equation*}
so that
\begin{equation*}
	\lA\widetilde{F}_2\rA_{Y^{\sigma+\delta}}\leq C\lb1+\lA\alpha-\frac{h^2}{16}\rA_{X^{s_0-\frac{1}{2}}}+\lA\beta\rA_{X^{s_0-\frac{1}{2}}}\rb
    \lA\nabla_{\!x,z}\widetilde{v}\rA_{X^\sigma}.
\end{equation*}
The estimates~(\ref{eq:alphasob}) and~(\ref{eq:betasob}) enable us to conclude once again.
\end{proof}

We can now decouple the equation into a forward and a backward parabolic evolution equation.
\begin{lem} \label{lem:estrellsob3}
There exists two symbols $a,A$ in $\Gamma^1_\eps(\R^d\times I)$ and a remainder $\widetilde{F}_3$ such that
\begin{equation} \label{eq:couppar}
(\partial_z-T_a)(\partial_z-T_A)\widetilde{v}=\widetilde{F}_0+\widetilde{F}_1+\widetilde{F}_2+\widetilde{F}_3,
\end{equation}
with
\begin{equation} \label{eq:symbpar}
\mathcal{M}^1_\eps(a)+\mathcal{M}^1_\eps(A)\leq\mathcal{F}\left(\lA\eta\rA_{W^{1,\infty}}\right)\lA\eta\rA_{C^{1+\eps}_*},
\end{equation}
and  for all $J\subset I$
\begin{equation*}
\lA\widetilde{F}_3\rA_{Y^{\sigma+\delta}(J)}\leq
      \mathcal{F}\left(\lA\eta\rA_{W^{1,\infty}}\right)\lA\eta\rA_{C^{1+\eps}_*}\lA\nabla_{\!x,z}\widetilde{v}\rA_{X^{\sigma}(I)}.
\end{equation*}
\end{lem}
\begin{proof}
We see from \cite{AlazardBurqZuilyExw/ST} that this holds true with 
\begin{equation}
a=\frac{1}{2}(i\beta\cdot\xi-\sqrt{4\alpha\lvert\xi\rvert^2-(\beta\cdot\xi)^2}),\quad
    A=\frac{1}{2}(-i\beta\cdot\xi+\sqrt{4\alpha\lvert\xi\rvert^2-(\beta\cdot\xi)^2}).
\end{equation}
Using the H\"{o}lder estimates (\ref{eq:alphahold}) and (\ref{eq:betahold}) gives (\ref{eq:symbpar}).
A straightforward estimate of the remainder along the lines of \cite{AlazardBurqZuilyExw/ST} ends the proof.
\end{proof}

\begin{proof}[Proof of Proposition \ref{prop:regsobboot}]
Still following \cite{AlazardBurqZuilyExw/ST}, we apply proposition \ref{prop:parabsob} twice.
We will prove only inequality~(\ref{eq:elliphold}), since the proof of inequality~(\ref{eq:ellipsob}) is along the same lines.

Suppose first that $\mathcal{H}_\sigma$ is satisfied. This means that for $J_0=[\zeta_0,0]$, we have
\begin{multline*} 
\lA\nabla_{\!x,z}\widetilde{v}\rA_{X^{\sigma}([\zeta_0,0])}\leq\mathcal{F}\left(\lA\eta\rA_{(C^{1+\eps}_*\cap L^2)(\R^d)}\right)\left[\lA f\rA_{H^{\sigma+1}(\R^d)}
	+\lA \widetilde{F}_0\rA_{Y^{\sigma}(I)}\right.\\
    \left.+\lA\eta\rA_{H^{s+\frac{1}{2}}(\R^d)}\lA\nabla_{\!x,z}\widetilde{v}\rA_{L^\infty(I\times\R^d)}
	+\lA\widetilde{v}\rA_{X^{-\frac{1}{2}}([-1,0])}\right]. 
\end{multline*}
We will then prove that, for any $\zeta_1>\zeta_0$, we have
\begin{multline*} 
\lA\nabla_{\!x,z}\widetilde{v}\rA_{X^{\sigma+\delta}([\zeta_1,0])}\leq\mathcal{F}(\lA\eta\rA_{(C^{1+\eps}_*\cap L^2)(\R^d)})\lb\lA f\rA_{H^{\sigma+1+\delta}(\R^d)}
	+\lA \widetilde{F}_0\rA_{Y^{\sigma+\delta}(J_0)}\right.\\
    \left.+\lA\eta\rA_{H^{s+\frac{1}{2}}(\R^d)}\lA\nabla_{\!x,z}\widetilde{v}\rA_{L^\infty(I\times\R^d)}
	+\lA\nabla_{\!x,z}\widetilde{v}\rA_{X^{\sigma}(J_0)}\rb, 
\end{multline*}
which will finish the proof of the proposition.

Take $\chi$ a cutoff function such that $\chi(\zeta_0)=0$ and $\chi(z)=1$ for $z\geq\zeta_1$. We then put $\widetilde{w}:=\chi(z)(\partial_z-T_A)\widetilde{v}$. We see from (\ref{eq:couppar}) that
$$\partial_z\widetilde{w}-T_a\widetilde{w}=\widetilde{F}',$$
with
$$\widetilde{F}'=\chi(z)(\widetilde{F}_0+\widetilde{F}_1+\widetilde{F}_2+\widetilde{F}_3)+\chi'(z)(\partial_z-T_A)\widetilde{v}.$$
All those terms, the last one excepted, have already been estimated.
Since $\delta\leq1/2$, a simple computation gives 
\begin{equation*} 
\lA(\partial_z-T_A)\widetilde{v}\rA_{Y^{\sigma+\delta}(J_0)}\leq\lA(\partial_z-T_A)\widetilde{v}\rA_{X^{\sigma}(J_0)}\leq
  \mathcal{F}\left(\lA\eta\rA_{W^{1,\infty}(\R^d)}\right)\lA\eta\rA_{C^{1+\eps}_*}\lA\nabla_{\!x,z}\widetilde{v}\rA_{X^{\sigma}(J_0)},
\end{equation*}
where we have used the fact that $A$ is a symbol of order 1, whose norm has already been estimated in (\ref{eq:symbpar}). 

We see from the definition of $a$ that it is elliptic, with an ellipticity constant depending only on $h$. Since we have $w|_{z=z_0}=0$, Proposition \ref{prop:parabsob} gives the estimate
\begin{align*}
  \lA\widetilde{w}\rA_{X^{\sigma+\delta}(I_0)}&\leq\mathcal{F}\left(\lA\eta\rA_{C^{1+\eps}(\R^d)}\right)\lA\widetilde{F}'\rA_{Y^{\sigma+\delta}(J_0)}\\
    &\leq\mathcal{F}\left(\lA\eta\rA_{(C^{1+\eps}_*\cap L^2)(\R^d)}\right)\left[\lA\widetilde{F}_0\rA_{Y^{\sigma+\delta}(J_0)}\right.\\
	&\left.+\lA\eta\rA_{H^{s+\frac{1}{2}}(\R^d)}\lA\nabla_{\!x,z}\widetilde{v}\rA_{L^\infty(I\times\R^d)}
	+\lA\nabla_{\!x,z}\widetilde{v}\rA_{X^{\sigma}(J_0)}\right],
\end{align*}
where we have used the estimates from Lemma \ref{lem:estrellsob1}, \ref{lem:estrellsob2} and \ref{lem:estrellsob3}.
Notice that on $J_1:=[\zeta_1,0]$, we have $\chi=1$, so that
$$\partial_z\widetilde{v}-T_A\widetilde{v}=\widetilde{w}.$$
In fact, we have 
$$\partial_z\nabla_{\!x}\widetilde{v}-T_A\nabla_{\!x}\widetilde{v}=\nabla_{\!x}\widetilde{w}+T_{\nabla_{\!x}A}\widetilde{v}.$$
Changing $z$ to $-z$, and using again Proposition \ref{prop:parabsob}, we have 
$$\lA\nabla_{\!x}\widetilde{v}\rA_{X^{\sigma+\delta}(J_1)}\leq\mathcal{F}\lp\lA\eta\rA_{C^{1+\eps}(\R^d)}\rp\lb\lA\nabla_{\!x}f\rA_{H^{\sigma+\delta}(\R^d)}
      +\lA\nabla_{\!x}\widetilde{w}\rA_{Y^{\sigma+\delta}(J_1)}+\lA\nabla_{\!x}\widetilde{v}\rA_{X^{\sigma}(J_0)}\rb.$$
Since $\lA\nabla_{\!x}\widetilde{w}\rA_{Y^{\sigma+\delta}(J_1)}\leq\lA\widetilde{w}\rA_{X^{\sigma+\delta}(J_1)}$, we obtain the estimate
\begin{multline} 
\lA\nabla_{\!x}\widetilde{v}\rA_{X^{\sigma+\delta}([\zeta_1,0])}\leq\mathcal{F}\left(\lA\eta\rA_{(C^{1+\eps}_*\cap L^2)(\R^d)}\right)\left[\lA\nabla_{\!x}f\rA_{H^{\sigma+1+\delta}(\R^d)}
	+\lA \widetilde{F}_0\rA_{Y^{\sigma+\delta}(J_0)}\right.\\
    \left.+\lA\eta\rA_{H^{s+\frac{1}{2}}(\R^d)}\lA\nabla_{\!x,z}\widetilde{v}\rA_{L^\infty(I\times\R^d)}
	+\lA\nabla_{\!x,z}\widetilde{v}\rA_{X^{\sigma}(J_0)}\right]. 
\end{multline}

The estimate for $\partial_z\widetilde{v}$ follows from $\partial_z\widetilde{v}=T_A\widetilde{v}+\widetilde{w}$, the estimate for $w$ and the fact that $A$ is of order 1.
This completes the proof of Proposition \ref{prop:regsobboot}, and hence of Proposition \ref{prop:regsob}.
\end{proof}

\subsection{Applications} \label{subsec:apl}
In this section we apply the previous elliptic estimates to study the Dirichlet-Neumann operator and its paralinearization.
\begin{prop} \label{prop:DN}
Let~$d\geq1$, $s>1+d/2$, $s_0>1/2+d/2$, and~$1/2\leq\sigma\leq s+1/2$.
Then
\begin{equation*}
\lA G(\eta)f\rA_{H^{\sigma-1}}\leq\mathcal{F}\lp\lA\eta\rA_{C^{1+\eps}_*}\rp\left[\lA f\rA_{H^\sigma}+\lA\eta\rA_{H^{s+\mez}}\lA\nabla_{x,z}v\rA_{L^\infty}\right],
\end{equation*}
and
\begin{equation*}
\lA G(\eta)f\rA_{H^{\sigma-1}}\leq\mathcal{F}\lp\lA\eta\rA_{H^{s_0+\mez}}\rp\lA f\rA_{H^\sigma},
\end{equation*}
where~$v$ is the harmonic extension of~$f$.
We also have
$$\lA G(\eta)f\rA_{L^\infty}\leq\mathcal{F}\lp\lA\eta\rA_{W^{1,\infty}}\rp\lA\nabla_{x,z}v\rA_{L^\infty}.$$
\end{prop}
\begin{rem}
The term~$\lA\nabla_{x,z}v\rA_{L^\infty}$ can generally be expressed only with terms defined on~$\Sigma$, using the maximum principle of Proposition~\ref{prop:maxprinc} and its consequence, 
Proposition~\ref{prop:maxprincder}.
\end{rem}
\begin{proof}
As seen in~\cite{AlazardBurqZuilyExw/ST}, the Dirichlet-Neumann can be expressed by
\begin{equation} \label{eq:DN}
	G(\eta)f=\left.\frac{1+\la\nabla_{\!x}\rho\ra^2}{\partial_z\rho}\partial_zv-\nabla_{\!x}\rho\cdot\nabla_{\!x}v\ra_{z=0}.
\end{equation}

Now using Proposition~\ref{prop:regsob}, the estimates on~$\rho$ from (\ref{eq:rhosob}), and the tame estimates (\ref{eq:tamestsob}), the first two estimates of the proposition immediately follow.
The last one is a straightforward consequence of~(\ref{eq:DN}).
\end{proof}

We know from~\cite{AlazardBurqZuilyExw/ST} that the Dirichlet-Neumann can be expressed as 
\begin{equation} \label{eq:paraDN}
G(\eta)=T_\lambda+R(\eta),
\end{equation}
where 
\begin{equation}
\lambda(x,\xi):=\sqrt{(1+\la\nabla\eta(x)\ra^2)\la\xi\ra^2-(\nabla\eta(x)\cdot\xi)^2},
\end{equation}
and $R(\eta)$ is a smoothing operator.
Using tame estimates, we obtain
\begin{prop} \label{prop:remDN}
  Let~$d\geq1$, $s>1+d/2$, $0<\eps\leq\eps'\leq1/2$ and~$1/2\leq\sigma\leq s+1/2$.
Then 
\begin{equation*}
\lA R(\eta)f\rA_{H^{\sigma-1+\eps'}}\leq
\mathcal{F}\lp\lA\eta\rA_{C^{1+\eps}_*}\rp\lp1+\lA\eta\rA_{C^{1+\eps'}_*}\rp\left[\lA f\rA_{H^\sigma}+\lA\eta\rA_{H^{s+\mez}}\lA\nabla_{x,z}v\rA_{L^\infty}\right],
\end{equation*}
and if~$s_0>1/2+d/2$ and~$\eps'\leq s_0-d/2-1/2$,
\begin{equation*}
\lA R(\eta)f\rA_{H^{\sigma-1+\eps'}}\leq
\mathcal{F}\lp\lA\eta\rA_{H^{s_0+\mez}}\rp\lA f\rA_{H^\sigma}.
\end{equation*}
\end{prop}

\begin{proof}
As in the proof of Proposition~\ref{prop:regsobboot}, we use Proposition~\ref{prop:parabsob} to get
$$\lA(\partial_z-T_A)\widetilde{v}\rA_{X^{\sigma-1+\eps'}(J_0)}\leq\mathcal{F}\left(\lA\eta\rA_{C^{1+\eps}(\R^d)}\right)\lA\widetilde{F}'\rA_{Y^{\sigma-1+\eps'}(J_0)}$$
Now since
$$G(\eta)f=\left.\frac{1+\la\nabla_{\!x}\rho\ra^2}{\partial_z\rho}\partial_zv-\nabla_{\!x}\rho\cdot\nabla_{\!x}v\ra_{z=0},$$
we set
$$\zeta_1:=\frac{1+\la\nabla_{\!x}\rho\ra^2}{\partial_z\rho},\quad\zeta_2:=\nabla_{\!x}\rho.$$
According to~(\ref{eq:rhosob}), 
\begin{equation}\label{3.56bis}
\lA \zeta_1-\frac{4}{h}\rA_{C([-1,0];H^{s-\mez})}+
\lA \zeta_2\rA_{C([-1,0];H^{s-\mez})}\leq \mathcal{F}(\lA\eta\rA_{W^{1,\infty}})\lA\eta\rA_{H^{s+\frac{1}{2}}}.
\end{equation}
Let 
$$
R':=\zeta_1\partial_z v
-\zeta_2 \cdot\nabla_{\!x} v -(T_{\zeta_1}  \partial_z v  
- T_{\zeta_2}\nabla_{\!x} v).
$$
Using the tame estimates of~(\ref{eq:tamestsob}), we verify that 
$$
\lA R'\rA_{C^0(I;H^{\sigma-1+\eps'})}\leq\mathcal{F}(\lA\eta\rA_{W^{1,\infty}})\lA\partial_zv\rA_{L^\infty}\lA\eta\rA_{H^{s+\mez}}.
$$
Furthermore, 
\begin{equation*}
T_{\zeta_1}  \partial_z v  
- T_{\zeta_2}\partial_x v\rvert_{z=0}
-(T_{\zeta_1}   T_A  v -T_{i\zeta_2\cdot \xi}  v) \rvert_{z=0} :=R'',
\end{equation*}
with 
~$$\lA R''\rA_{H^{\sigma-1+\eps'}}\leq \mathcal{F}\lp\lA\eta\rA_{C^{1+\eps}_*}\rp\lp1+\lA\eta\rA_{C^{1+\eps'}_*}\rp\left[\lA f\rA_{H^\sigma}+\lA\eta\rA_{H^{s+\mez}}\lA\nabla_{x,z}v\rA_{L^\infty}\right].$$
Finally,  we have
\begin{equation*}
\lA T_{\zeta_1(z)}   T_{A(z)} -T_{\zeta_1(z) A(z)}\rA_{ H^{\sigma}\rightarrow H^{\sigma-1+\eps'}}
\leq \mathcal{F}(\lA\eta\rA_{W^{1,\infty}})\lA\eta\rA_{C^{1+\eps'}_*},
\end{equation*}
and hence
$$
G(\eta)f=T_{\zeta_1 A}  v -T_{i\zeta_2\cdot \xi}  v \big\arrowvert_{z=0} +R(\eta)f,
$$
where 
$$\lA R(\eta)f\rA_{H^{\sigma-1+\eps'}}\leq
\mathcal{F}(\lA\eta\rA_{C^{1+\eps}_*})\lp1+\lA\eta\rA_{C^{1+\eps'}_*}\rp\left[\lA f\rA_{H^\sigma}+\lA\eta\rA_{H^{s+\mez}}\lA\nabla_{x,z}v\rA_{L^\infty}\right].$$
Let 
\begin{equation*}
\lambda:=\frac{1 +|\partial_x \rho |^2}{\partial_z \rho} A -i  \partial_x \rho
\cdot  \xi\big\arrowvert_{z=0}= \sqrt{(1+|\partial_x \eta(x)|^2)| \xi |^2
-\bigl(\partial_x \eta (x)\cdot \xi\bigr)^2}.
\end{equation*}
Then 
$$
G(\eta)f=T_{\lambda} f +R(\eta)f,
$$
which concludes the proof of the first inequality. The second one is proved along the same lines.
\end{proof}

\section{Paralinearization of the system} \label{sec:Paralin}

We still follow~\cite{AlazardBurqZuilyExw/ST} to reduce the equations to a paradiferential hyperbolic system.
This process yields remainders that we need to estimate. To simplify the expression of these estimates, we will denote by~$K$ a constant of the form
$$K(t):=\mathcal{F}\left(\frac{1}{h},\;\frac{1}{c},\;
\lA\eta(t)\rA_{W^{1+\eps,\infty}(\R^d)},\;
  \lA a(t)\rA_{W^{\eps,\infty}(\R^d)},\;\lA\eta\rA_{L^2(\R^d)},\;\lA(V,B)\rA_{W^\eps(\R^d)}\right),$$ 
  with~$\mathcal{F}$ positive nondecreasing. It will appear in the proof of the first version of our main theorem, Theorem~\ref{thm:mainhold}, involving only H\"{o}lder norms.
  For the proof of the second version, Theorem~\ref{thm:mainsob}, involving the reference Sobolev norm of index~$s_0>1/2+d/2$, we will use a constant
\begin{equation*}
	K_0(t):=\mathcal{F}\lp\frac{1}{h},\frac{1}{c},\lA a(t)\rA_{W^{\eps,\infty}(\R^d)},
	    \lA\lp\eta,\psi,V,B\rp(t)\rA_{H^{s_0+\mez}(\R^d)\times H^{s_0+\mez}(\R^d)\times H^{s_0}(\R^d)\times H^{s_0}(\R^d)}\rp.
\end{equation*}
Using Sobolev embeddings, we see that we can take~$K\leq K_0$. 

We still denote by $C$ a generic constant. 
To get the optimal regularity, we need to change the unknowns we are working with, using instead
$$\zeta=\nabla\eta,\quad B=\partial_y\phi\rvert_{y=\eta},\quad V=\nabla_{\!x}\phi\rvert_{y=\eta},\quad a=-\partial_yP\rvert_{y=\eta},$$
where~$\phi$ is the velocity potential and~$P$ the pressure, uniquely determined by the equation
$$-P=\partial_t\phi+\mez\la\nab\phi\ra^2+gy.$$

Those follow the following evolution equations.
\begin{prop} \label{prop:evoleq}
	We have
	\begin{align}
		\left(\partial_t+V\cdot\nabla\right)B&=a-g, \label{eq:Beq}\\
		\left(\partial_t+V\cdot\nabla\right)V&=-a\zeta, \label{eq:Veq}\\
		\left(\partial_t+V\cdot\nabla\right)\zeta&=G(\eta)V+\zeta G(\eta)B+\gamma, \label{eq:zetaeq}
	\end{align}
	where the remainder~$\gamma$ satisfies the estimates
	\begin{equation} \label{eq:rembothold}
		\lA\gamma\rA_{H^{s-\mez}}\leq K\left[\lA(V,B,\psi)\rA_{H^{s-\mez}\times H^{s-\mez}\times H^{s+\mez}}+\lA\eta\rA_{H^{s+\mez}}\lA(V,B)\rA_{C_*^{1+\eps}}\right],
	\end{equation}
	and
	\begin{equation} \label{eq:rembotsob}
	\lA\gamma\rA_{H^{s-\mez}}\leq K_0\left[\lA(V,B,\psi)\rA_{H^{s-\mez}\times H^{s-\mez}\times H^{s+\mez}}\right]	.
	\end{equation}

\end{prop}
\begin{proof}
	The first two equations are proved in~\cite{AlazardBurqZuilyExw/ST}. 
	For the estimations of ~$\gamma$, we start from the equation
	$$\partial_t\eta=B-V\cdot\nabla\eta.$$
	Differentiating this with respect to~$x_i$ yields for~$i=1,...,d$
	$$\left(\partial_t+V\cdot\nabla\right)\partial_i\eta=\partial_iB-\sum_{j=1}^d\partial_iV_j\partial_j\eta,$$
	and using the definitions of~$V$ and~$B$ and the chain rule,
	\begin{equation}
		\left(\partial_t+V\cdot\nabla\right)\partial_i\eta=\left[\partial_y\partial_i\phi-\nabla\eta\cdot\nabla\partial_i\phi\right]\rvert_{y=\eta}
		+\partial_i\eta\left[\partial_y(\partial_y\phi)-\nabla\eta\cdot\nabla\partial_y\phi\right]\rvert_{y=\eta}.
	\end{equation}
	
	We now introduce~$\theta_i$, the variational solution of the problem
	$$\Delta_{x,y}\theta_i=0\text{ in }\Omega,\quad\theta_i\rvert_\Sigma=V_i,\quad\partial_n\theta_i=0\text{ on }\Gamma.$$
	Then 
	$$G(\eta)V_i=\left(\partial_y\theta_i-\nabla\eta\cdot\nabla\theta_i\right)\rvert_\Sigma.$$
	We can now write 
	$$\left(\partial_y-\nabla\eta\cdot\nabla\right)\partial_i\phi\rvert_{y=\eta}=G(\eta)V_i+R_i,
	\quad\text{where }R_i=\left(\partial_y-\nabla\eta\cdot\nabla\right)\left(\partial_i\phi-\theta_i\right)\rvert_\Sigma.$$
	If there was no bottom, we would see that at least formally, $R_i$ would be 0. Then, in our setting, we expect a control of this remainder, and to obtain it, 
	we continue to follow~\cite{AlazardBurqZuilyExw/ST} and localize the problem near~$\Sigma$.
	
	Let~$\chi_0\in C^\infty(\R)$, $\eta_1\in H^\infty(\R^d)$ be such that~$\chi_0(z)=1$ if~$z\geq0$, $\chi_0(z)=0$ if~$z\leq-1/4$, and
	$$\eta(x)-\frac{h}{4}\leq\eta_1(x)\leq\eta(x)-\frac{h}{5}.$$
	Set
	$$U_i(x,y)=\chi_0\left(\frac{y-\eta_1(x)}{h}\right)\left(\partial_i\phi-\theta_i\right)(x,y).$$
	We see that
	$$R_i=\left(\partial_y-\nabla\eta\cdot\nabla\right)U_i\rvert_\Sigma.$$
	And~$U_i$ satisfies the equation
	$$\Delta_{x,y}U_i=\left[\Delta_{x,y},\chi_0\left(\frac{y-\eta_1(x)}{h}\right)\right]\left(\partial_i\phi-\theta_i\right):=F_i$$
	with
	$$\text{supp}F_i\subset S_{\frac{1}{2},\frac{1}{5}}:=\left\lbrace(x,y):x\in\R^d,\eta(x)-\frac{h}{2}\leq y\leq\eta(x)-\frac{h}{5}\right\rbrace.$$
	We can then control the right hand term of this equation, using lemma 3.16 of~\cite{AlazardBurqZuilyExw/ST}, which gives for all~$\alpha\in\N^{d+1}$,
	\begin{equation} \label{eq:remtr}
		\lA D_{x,y}^\alpha F_i\rA_{L^\infty(S_{\frac{1}{2},\frac{1}{5}})\cap L^2(S_{\frac{1}{2},\frac{1}{5}})}\leq C_\alpha\lA(V,B)\rA_{H^\mez\times H^\mez}.
	\end{equation}
	Then changing variables, we get on the domain~$\widetilde{\Omega}$ that
	$$\left(\partial_z^2+\alpha\Delta+\beta\cdot\nabla\partial_z-\gamma\partial_z\right)\widetilde{U}_i=\frac{(\partial_z\rho)^2}{1+\la\nabla\rho\ra^2}\widetilde{F}_i.$$
	We can then apply the estimate~(\ref{eq:elliphold}), with~$f=0$. 
	$$\lA\nabla_{\!x,z}\widetilde{U}_i\rA_{C^0([z_0,0];H^{s-\mez})}\leq K\left[\lA\widetilde{F}_i\rA_{Y^{s-\mez}}+
	\lA\eta\rA_{H^{s+\mez}}\lA\nabla_{\!x,z}\widetilde{U}_i\rA_{L^\infty}+\lA\nabla_{\!x,z}\widetilde{U}_i\rA_{X^{-\mez}}\right].$$
	Using equation~\ref{eq:remtr}, the control on the $X^{-\mez}$ norm of a variational solution, and the maximum principle for gradients of Proposition~\ref{prop:maxprincder}, this yields
	$$\lA\nabla_{\!x,z}\widetilde{U}_i\rA_{C^0([z_0,0];H^{s-\mez})}\leq K\left[\lA(V,B,\psi)\rA_{H^{s-\mez}\times H^{s-\mez}\times H^{s+\mez}}
	+\lA\eta\rA_{H^{s+\mez}}\lA(V,B)\rA_{C_*^{1+\eps}}\right].$$
	Since 
	$$R_i=\left.\left[\left(\frac{1+\la\nabla\eta\ra^2}{1+\delta\langle D_x\rangle\eta}\partial_z-\nabla\eta\cdot\nabla\right)\widetilde{U}_i\right]\ra_{z=0},$$
	we have 
	$$\lA R_i\rA_{H^{s-\mez}}\leq K\left[\lA(V,B,\psi)\rA_{H^{s-\mez}\times H^{s-\mez}\times H^{s+\mez}}+\lA\eta\rA_{H^{s+\mez}}\lA(V,B)\rA_{C_*^{1+\eps}}\right].$$
	The same argument shows that 
	$$\left(\partial_y-\nabla\eta\cdot\nabla\right)\partial_y\phi=G(\eta)B+R_0,$$
	where~$R_0$ satisfies the same estimate as~$R_i$. This proves the first estimate. The second one follows exactly the same scheme, using~(\ref{eq:ellipsob}) instead.
\end{proof}

Using the same method, and following proposition~4.5 of~\cite{AlazardBurqZuilyExw/ST}, it is possible to prove the following relation.
\begin{prop} \label{prop:linkVB}
	We have $G(\eta)B=-\text{div}\;V+\gamma'$, where
	$$\lA\gamma'\rA_{H^{s-\mez}}\leq K\left[\lA(V,B)\rA_{H^{s-\mez}\times H^{s-\mez}}+\lA\eta\rA_{H^{s+\mez}}\lA(V,B)\rA_{C_*^{1+\eps}}\right],$$
	and
	\begin{equation*}
		\lA\gamma'\rA_{H^{s-\mez}}\leq K_0\left[\lA(V,B)\rA_{H^{s-\mez}\times H^{s-\mez}}\right].
	\end{equation*}

\end{prop}

We will now perform a paralinearization of the system. We will start with the estimate for the first theorem.
We introduce as a new unknown 
$$U:=V+T_\zeta B.$$
Rather than estimating~$U$ and ~$\zeta$ in Sobolev spaces, it will be easier to estimate
\begin{align}
	U_s&:=\jap^sV+T_\zeta\jap^sB,\\
	\zeta_s&:=\jap^s\zeta.
\end{align}

\begin{prop}
	We have 
	\begin{equation} \label{eq:syspara}
		\left\lbrace
		\begin{gathered}
			\left(\partial_t+T_V\cdot\nabla\right)U_s+T_a\zeta_s=f_1,\\
			\left(\partial_t+T_V\cdot\nabla\right)\zeta_s=T_\lambda U_s+f_2,
		\end{gathered}
		\right.
	\end{equation}
	where~$\lambda$ is the symbol
	$$\lambda(t;x,\xi):=\sqrt{\left(1+\la\nabla\eta(t,x)\ra^2\right)\la\xi\ra^2-\left(\nabla\eta(t,x)\cdot\xi\right)^2},$$
	and where
	\begin{equation} \label{eq:remhold1}
	\begin{aligned}
		\lA f_1\rA_{L^2}\leq K&\left\lbrace\left[\lA\nabla\eta\rA_{C^\mez_*}\lA\nabla B\rA_{C_*^{-\mez}}+\lA\nabla B\rA_{L^\infty}+\lA V\rA_{W^{1,\infty}}\right]\lA V\rA_{H^s}\right.\\
						   &+\left[1+\lA a\rA_{C^\mez_*}\right]\lA\zeta\rA_{H^{s-\mez}}\\
						   &+\left[\lA V\rA_{C^{1+\eps}_*}+\lA\nabla B\rA_{L^\infty}\right]\lA B\rA_{H^s}\\
						   &\left.+\lA\nabla\eta\rA_{C^\mez_*}\lA a-g\rA_{H^{s-\mez}}\right\rbrace,
	\end{aligned}
	\end{equation}
	and
	\begin{equation} \label{eq:remhold2}
	\begin{aligned}
		\lA f_2\rA_{H^{-\mez}}\leq K&\left[\lp1+\lA\nabla\eta\rA_{C^\mez_*}\rp\lp\lA B\rA_{H^s}+\lA V\rA_{H^s}\rp\right.\\
					     &\left.+\lp\lA B\rA_{C^{1+\eps}_*}+\lA V\rA_{C^{1+\eps}_*}\rp\lA\eta\rA_{H^{s+\mez}}+\lA\psi\rA_{H^{s+\mez}}\right].
	\end{aligned}
	\end{equation}

\end{prop}
\begin{proof}
	The computations are long, however they still mirror the ones of~\cite{AlazardBurqZuilyExw/ST}.
	First we paralinearize the equation
	$$\left(\partial_t+V\cdot\nabla\right)V+a\zeta=0.$$
	We will prove the identity
	$$\left(\partial_t+T_V\cdot\nabla\right)V+T_a\zeta+T_\zeta\left(\partial_t+T_V\cdot\nabla\right)B=h_1$$
	with a remainder~$h_1$ satisfying
	$$\lA h_1\rA_{H^s}\leq K\left[\lA\nabla V\rA_{L^\infty}\lA V\rA_{H^s}+\left(1+\lA a\rA_{C^\mez_*}\right)\lA\zeta\rA_{H^{s-\mez}}+\lA\nabla B\rA_{L^\infty}\lA V\rA_{H^s}\right].$$
	First we have~$V\cdot\nabla V=T_V\cdot\nabla V+A_1$ with
	$$\lA A_1\rA_{H^s}\leq C\lA\nabla V\rA_{L^\infty}\lA V\rA_{H^s}.$$
	We also have $(a-g)\zeta=T_{a-g}\zeta+T_\zeta(a-g)+R(\zeta,a-g)$ where 
	$$\lA R(\zeta,a-g)\rA_{H^s}\leq C\lA\zeta\rA_{H^{s-\mez}}\lA a\rA_{C^\mez_*}.$$
	We can now replace~$a$ by~$g+\left(\partial_t B+V\cdot\nabla B\right)$ to get
	$$T_\zeta a=T_\zeta\lp\partial_tB+T_V\cdot\nabla B\rp+T_\zeta\lp V-T_V\rp\cdot\nabla B,$$
	with 
	$$\lA T_\zeta\lp V-T_V\rp\cdot\nabla B\rA_{H^s}\leq C\lA\nabla\eta\rA_{L^\infty}\lA\nabla B\rA_{L^\infty}\lA V\rA_{H^s}.$$
	
	We then commute our identity with~$\jap^s$.
	Using the product estimates of Theorem~\ref{thm:para}, we have the estimates
	\begin{gather}
		\lA\left[T_V\cdot\nabla,\jap^s\right]V\rA_{L^2}\leq C\lA V\rA_{W^{1,\infty}}\lA V\rA_{H^s},\\
		\lA\left[T_a,\jap^s\right]\zeta\rA_{L^2}\leq C\lA a\rA_{C^\mez_*}\lA\zeta\rA_{H^{s-\mez}},
	\end{gather}
	\begin{multline}
		\lA\left[T_\zeta,\jap^s\right]\lp\partial_t+T_V\cdot\nabla\rp B\rA_{L^2}\leq C\lA\zeta\rA_{C^\mez_*}\lA\lp\partial_t+T_V\cdot\nabla\rp B\rA_{H^{s-\mez}}\\
											   \leq C\lA\zeta\rA_{C^\mez_*}\left[\lA\lp\partial_t+V\cdot\nabla\rp B\rA_{H^{s-\mez}}
													      +\lA\lp V-T_V\rp\cdot\nabla B\rA_{H^{s-\mez}}\right]\\
											   \leq C\lA\zeta\rA_{C^\mez_*}\left[\lA a-g\rA_{H^{s-\mez}}+\lA\nabla B\rA_{C^{-\mez}_*}\lA V\rA_{H^s}\right],		
	\end{multline}
	\begin{gather}
		\lA T_\zeta\left[T_V\cdot\nabla,\jap^s\right]B\rA_{L^2}\leq C\lA\zeta\rA_{L^\infty}\lA V\rA_{W^{1,\infty}}\lA B\rA_{H^s},\\
		\lA\left[T_\zeta,\partial_t+T_V\cdot\nabla\right]\jap^sB\rA_{L^2}\leq C\left[\lA\nabla B\rA_{L^\infty}+\lA\nabla\eta\rA_{L^\infty}\lA V\rA_{C^{1+\eps}_*}\right]\lA B\rA_{H^s}.
	\end{gather}
	Those commutators estimates prove that
	\begin{equation}
		\lp\partial_t+T_V\cdot\nabla\rp\lp\jap^sV+T_\zeta\jap^sB\rp+T_a\jap^s\zeta=f_1,
	\end{equation}
	where~$f_1$ satisfies~(\ref{eq:remhold1}).
	
	We now paralinearize the equation 
	$$\lp\partial_t+V\cdot\nabla\rp\zeta=G(\eta)V+\zeta G(\eta)B+\gamma.$$
	We use the paralinearization of the Dirichlet-Neumann~(\ref{eq:paraDN}) to get
	$$\lp\partial_t+T_V\cdot\nabla\rp\zeta=T_\lambda U+h_2$$
	with
	$$h_2=-\lp V-T_V\rp\cdot\nabla\zeta+\left[T_\zeta,T\lambda\right]B+R(\eta)V+\zeta R(\eta)B+\lp\zeta-T_\zeta\rp T_\lambda B+\gamma.$$
	From Theorem~\ref{thm:paraprod}, we get 
	$$\lA\lp V-T_V\rp\cdot\nabla\zeta\rA_{H^{s-\mez}}\leq C\lA\nabla\eta\rA_{C^\mez_*}\lA V\rA_{H^s}.$$
	Theorem~\ref{thm:para} and simple estimates on the symbol~$\lambda$ give
	\begin{align}
		\lA\left[T_\zeta,T_\lambda\right]B\rA_{H^{s-\mez}}&\leq C\lp\M_0^0(\zeta)\M^1_\mez(\lambda)+\M^0_\mez(\zeta)M^1_0(\lambda)\rp\lA B\rA_{H^s}\\
								   &\leq K\lA\nabla\eta\rA_{C^\mez_*}\lA B\rA_{H^s}.
	\end{align}
	Then with the estimates of Proposition~\ref{prop:remDN} and the maximum principle~(\ref{prop:maxprincder}), we have 
	$$\lA R(\eta)V\rA_{H^{s-\mez}}\leq K\lp\lA V\rA_{H^s}+\lA\eta\rA_{H^{s+\mez}}\lA\nabla V\rA_{C^{1+\eps}_*}\rp,$$
	and using paraproduct from Theorem~\ref{thm:paraprod}, a rough estimate of~$\lA R(\eta)B\rA_{L^\infty}$, and the maximum principle~(\ref{prop:maxprincder}),
	\begin{align}
		\lA\zeta R(\eta)B\rA_{H^{s-\mez}}&\leq C\lp\lA\zeta\rA_{H^{s-\mez}}\lA R(\eta)B\rA_{L^\infty}+\lA\zeta\rA_{L^\infty}\lA R(\eta)B\rA_{H^{s-\mez}}\rp\\
						  &\leq K\lp\lA\eta\rA_{H^s+\mez}\lA B\rA_{C^{1+\eps}_*}+\lA B\rA_{H^s}\rp.
	\end{align}
	At last, we see thanks to the estimates of Theorem~\ref{thm:para} that
	$$\lA\lp\zeta-T_\zeta\rp T_\lambda B\rA_{H^{s-\mez}}\leq K\lA B\rA_{C^{1+\eps}_*}\lA\zeta\rA_{H^{s-\mez}}.$$
	Then as in the previous part, commuting the equation with~$\jap^s$ and using~(\ref{eq:rembothold}) yields~(\ref{eq:remhold2}).
\end{proof}

In order to obtain a closed inequality system, we need an estimate of~$\lA a(t)-g\rA_{H^{s-\mez}}$ in terms of the Sobolev norms of the unknowns~$\eta,\psi,V,B$.
This is the object of the following proposition.
\begin{prop} \label{prop:Tay}
	The Taylor coefficient satisfies the estimates
	\begin{equation*}
		\lA a-g\rA_{H^{s-\mez}}\leq K\lb\lA\eta\rA_{H^{s+\mez}}\lp1+\lA(V,B)\rA_{C^{1+\eps}_*}\rp+\lA(V,B)\rA_{C^{1+\eps}_*}\lA\psi\rA_{H^{s+\mez}}\rb,
	\end{equation*}
	and
	\begin{equation*}
		\lA a-g\rA_{H^{s-\mez}}\leq K_0\lA\eta,\psi,V,B\rA_{H^{s+\mez}\times H^{s+\mez}\times H^s\times H^s}
	\end{equation*}

\end{prop}
\begin{proof}
	The pressure is defined by
	$$P=-\lp\partial_t\phi+\mez\la\nabla_{\!x}\phi\ra^2+\mez\lp\partial_y\phi\rp^2+gy\rp,$$
	where~$\phi$ is the harmonic extension of ~$\psi$. This means that~$P$ satisfies the elliptic equation
	$$\Delta_{x,y}P=-\la\nab^2\phi\ra^2,$$
	with $P=0$ on the free surface $\Sigma$. 
	We change variables using the transformation~$\rho$ from~(\ref{eq:defvar}), and set 
	$$\ph(x,z)=\phi(x,\rho(x,z)),\quad\P(x,z)=P(x,\rho(x,z))+g\rho(x,z),$$
	with 
	$$a-g=\left.-\frac{1}{\partial_z\rho}\partial_z\P\ra_{z=0}.$$
	The elliptic equation on~$P$ becomes
	\begin{align}
		&\partial_z^2\P+\alpha\Delta_x\P+\beta\cdot\nabla_{\!x}\partial_z\P-\gamma\partial_z\P=-\alpha\la\Lambda^2\ph\ra^2\quad&&\text{for }z<0,\\
		&\P=g\eta												       &&\text{on }z=0,
	\end{align}
	where~$\Lambda=(\Lambda_1,\Lambda_2)$ is defined in~(\ref{eq:Lambda}).
	
	We first need to study the right-hand term of the equation.
	Since~$\phi$ is harmonic, we recover from Proposition~\ref{prop:regsob} and the variational estimate of~$\lA\nabz\ph\rA_{X^{-\mez}}$ the inequality
	$$\lA\nabz\ph\rA_{X^{s-\mez}}\leq K\lb\lA\eta\rA_{H^{s+\mez}}+\lA\psi\rA_{H^{s+\mez}}\rb.$$
	Now using the fact that~$(\Lambda_1^2+\Lambda_2^2)\ph=0$, we can recover estimates on~$\partial_z^2\ph$ from estimates on~$\nabla_{\!x}\nabz\ph$, so that
	$$\lA\Lambda^2\ph\rA_{X^{s-\frac{3}{2}}}\leq K\lb\lA\eta\rA_{H^{s+\mez}}+\lA\psi\rA_{H^{s+\mez}}\rb.$$
	At last, using the paraproduct rules, and the estimates on~$\alpha$ from Lemma~\ref{lem:ellcoefs}, we find
	\begin{align}
	\lA-\alpha\la\Lambda^2\ph\ra^2\rA_{Y^{s-\mez}}\leq&K\lA\Lambda^2\ph\rA_{L^\infty}\lb\lA\eta\rA_{H^{s+\mez}}\lA\Lambda^2\ph\rA_{L^\infty([-1,0];C^{-1}_*)}+\lA\Lambda^2\ph\rA_{X^{s-\frac{3}{2}}}\rb\\
						       \leq&K\lA(V,B)\rA_{C^{1+\eps}_*}\lb\lA\eta\rA_{H^{s+\mez}}+\lA\psi\rA_{H^{s+\mez}}\rb,
	\end{align}
	where~$\lA\Lambda^2\ph\rA_{L^\infty([-1,0];C^{-1}_*)}$ has been estimated from~$\lA\Lambda\ph\rA_{L^\infty}$ using~$(\Lambda_1^2+\Lambda_2^2)\ph=0$ once again.
	
	For the version with a reference Sobolev index, we recall from the proof of Proposition~\ref{prop:evoleq} that if~$\theta_i$ is the harmonic extension of~$V_i$, and if~$\tau_i$ is its straightening
	by the diffeomorphism~$\rho$ to the strip, then up to a harmless restriction of the interval~$J$ close to the boundary we have
	\begin{equation*}
		\lA\nabz\lp\tau_i-\Lambda_i\ph\rp\rA_{X^{s-\mez}(J)}\leq K_0\lA(\psi,V,B)\rA_{H^{s+\mez}\times H^{s-\mez}\times H^{s-\mez}}.
	\end{equation*}
	Also, we have from~(\ref{eq:ellipsob})
	\begin{equation*}
		\lA\nabz\tau_i\rA_{X^{s-1}}\leq K_0\lA V\rA_{H^s}.
	\end{equation*}
	Combining those two results and doing the same for~$B$ and~$\partial_z\ph$ gives
	\begin{equation*}
		\lA\Lambda^2\ph\rA_{X^{s-1}}\leq K_0\lA\eta,\psi,V,B\rA_{H^{s+\mez}\times H^{s+\mez}\times H^s\times H^s}.
	\end{equation*}
	At last, using paraproduct estimates, we gain
	\begin{equation} \label{eq:sourcesob}
		\lA-\alpha\la\Lambda^2\ph\ra^2\rA_{Y^{s-\mez}}\leq K_0\lA\eta,\psi,V,B\rA_{H^{s+\mez}\times H^{s+\mez}\times H^s\times H^s}.
	\end{equation}

	We then take from~\cite{AlazardBurqZuilyExw/ST} the estimate
	\begin{align}
		\lA\nabz\P\rA_{X^{-\mez}}&\leq K\lb\lA\eta\rA_{H^\mez}+\lA\la\nabla\ph\ra^2\rA_{X^\mez}\rb\\
				       &\leq K\lb1+\lA(V,B)\rA_{C^{1+\eps}_*}\rb\lb\lA\eta\rA_{H^{s+\mez}}+\lA\psi\rA_{H^{s+\mez}}\rb.
	\end{align}
	A last application of the elliptic regularity of Proposition~\ref{prop:regsob} gives the estimate
	\begin{align}
		\lA\partial_z\P\rA_{X^{s-\mez}}&\leq K\lb\lA\eta\rA_{H^{s+\mez}}\lp1+\lA\partial_z\P\rA_{L^\infty}\rp+\lA-\alpha\la\Lambda^2\ph\ra^2\rA_{Y^{s-\mez}}+\lA\nabz\P\rA_{X^{-\mez}}\rb\\
					       &\leq K\lb\lA\eta\rA_{H^{s+\mez}}\lp1+\lA(V,B)\rA_{C^{1+\eps}_*}\rp+\lA(V,B)\rA_{C^{1+\eps}_*}\lA\psi\rA_{H^{s+\mez}}\rb,
	\end{align}
	and a last use of the paraproduct gives the first result.
	The second one follows along the same lines, using~(\ref{eq:ellipsob}) and~(\ref{eq:sourcesob}) instead.
\end{proof}

We can now perform a symmetrization of the system as follows
\begin{prop} \label{prop:parasyshold}
	We introduce the symbols
	$$\gamma:=\sqrt{a\lambda},\quad q:=\sqrt{\frac{a}{\lambda}},$$
	and the new variable 
	$$\theta_s:=T_q\zeta_s.$$
	Then we have the equations
	\begin{equation} \label{eq:syssym}
		\left\lbrace
		\begin{aligned}
			\partial_tU_s+T_V\cdot\nabla U_s+T_\gamma\theta_s&=F_1,\\
			\partial_t\theta_s+T_V\cdot\nabla\theta_s-T_\gamma U_s&=F_2,\\
		\end{aligned}
		\right.
	\end{equation}
	where the source terms~$F_1$ and~$F_2$ satisfy
	\begin{equation*}
		\begin{aligned}
			\lA F_1\rA_{L^2}\leq K&\lb\lp1+\lA\nabla\eta\rA_{C^\mez_*}\rp\lp1+\lA(V,B)\rA_{C^{1+\eps}_*}\rp\lA(V,B)\rA_{H^s}\right.\\
					      &+\lp1+\lA a\rA_{C^\mez_*}+\lp1+\lA(V,B)\rA_{C^{1+\eps}_*}\rp\lp1+\lA\nabla\eta\rA_{C^\mez_*}\rp\rp\lp1+\lA\zeta\rA_{H^{s-\mez}}\rp\\
					      &\left.+\lp1+\lA(V,B)\rA_{C^{1+\eps}_*}\rp\lp1+\lA\nabla\eta\rA_{C^\mez_*}\rp\lA\psi\rA_{H^{s+\mez}}\rb,
		\end{aligned}
	\end{equation*}
	and
	\begin{equation*}
		\begin{aligned}
			\lA F_2\rA_{L^2}\leq K&\lb\lp1+\lA\nabla\eta\rA_{C^\mez_*}\rp\lA(V,B)\rA_{H^s}\right.\\
					      &+\lp1+\lA(V,B)\rA_{C^{1+\eps}_*}+\lA\partial_ta+V\cdot\nabla a\rA_{L^\infty}\rp\lp1+\lA\zeta\rA_{H^{s-\mez}}\rp\\
					      &+\lA\psi\rA_{H^{s+\mez}}\\
					      &\left.+\lp1+\lA a\rA_{C^\mez_*}+\lA\nabla\eta\rA_{C^\mez_*}\rp\lA U_s\rA_{L^2}\rb.
		\end{aligned}
	\end{equation*}

\end{prop}
\begin{proof}
	We have from the preceding system~(\ref{eq:syspara}) that~(\ref{eq:syssym}) is satisfied for
	$$F_1=f_1+\lp T_\gamma T_q-T_a\rp\zeta_s,$$
	$$F_2=T_q f_2+\lp T_qT_\lambda-T_\gamma\rp U_s-\left[T_q,\partial_t+T_V\cdot\nabla\right]\zeta_s.$$
	
	Thanks to Lemma~\ref{lem:comconv}, we have 
	$$\lA\left[T_q,\partial_t+T_V\cdot\nabla\right]\zeta_s\rA_{L^2}\leq C\lb\M^{-\mez}_0(q)\lA V\rA_{C^{1+\eps}_*}+\M^{-\mez}_0(\partial_t q+V\cdot\nabla q)\rb\lA\zeta_s\rA_{H^{-\mez}}.$$
	It can be computed that 
	$$\M^{-\mez}_0(q)\leq K,$$
	and that
	$$\M^{-\mez}_0(\partial_t q+V\cdot\nabla q)\leq K\lp1+\lA\partial_ta+V\cdot\nabla a\rA_{L^\infty}+\lA\partial_t\nabla\eta+V\cdot\nabla\nabla\eta\rA_{L^\infty}\rp.$$
	A differentiation of the identity
	$$(\partial_t+V\cdot\nabla)\eta=B$$
	gives the estimate
	$$\lA(\partial_t+V\cdot q)\partial_{x_i}\eta\rA_{L^\infty}\leq\lA\nabla B\rA_{L^\infty}+\lA\nabla V\rA_{L^\infty}\lA\nabla\eta\rA_{L^\infty},$$
	so that 
	\begin{equation}
		\lA\left[T_q,\partial_t+T_V\cdot\nabla\right]\zeta_s\rA_{L^2}\leq K\lp\lA(V,B)\rA_{C^{1+\eps}_*}+\lA\partial_ta+V\cdot\nabla a\rA_{L^\infty}\rp\lA\zeta_s\rA_{H^{-\mez}}.
	\end{equation}
	The estimates of the other terms give respectively
	\begin{equation}
		\begin{aligned}
			\lA(T_\gamma T_q-T_a)\zeta_s\rA_{L^2}&\leq C\lb\M^\mez_\mez(\gamma)\M^{-\mez}_0(q)+\M^\mez_0(\gamma)\M^{-\mez}_\mez(q)\rb\lA\zeta_s\rA_{H^{-\mez}}\\
							      &\leq K\lb1+\lA a\rA_{C^\mez_*}+\lA\nabla\eta\rA_{C^\mez_*}\rb\lA\zeta_s\rA_{H^{-\mez}},
		\end{aligned}
	\end{equation}
	\begin{equation}
		\begin{aligned}
			\lA(T_qT_\lambda-T_\gamma)U_s\rA_{L^2}&\leq C\lb\M^{-\mez}_\mez(q)\M^1_0(\lambda)+\M^{-\mez}_0(q)\M^1_{\mez}(\lambda)\rb\lA U_s\rA_{L^2}\\
							       &\leq K\lb1+\lA a\rA_{C^\mez_*}+\lA\nabla\eta\rA_{C^\mez_*}\rb\lA U_s\rA_{L^2},
		\end{aligned}
	\end{equation}
	and lastly
	\begin{equation}
		\begin{aligned}
			\lA T_qf_2\rA_{L^2}&\leq C\M^{-\mez}_0(q)\lA f_2\rA_{H^{-\mez}}\\
					    &\leq K\lA f_2\rA_{H^{-\mez}}.
		\end{aligned}
	\end{equation}
	This, together with the previous estimates, give the expected result.
\end{proof}

The analogous result with the reference Sobolev index is 
\begin{prop} \label{prop:parasyssob}
	The source terms~$F_1$ and~$F_2$ from the preceding proposition satisfy
	\begin{equation*}
		\begin{aligned}
			\lA F_1\rA_{L^2}\leq K_0&\lb\lp1+\lA\nabla\eta\rA_{C^\mez_*}+\lA(V,B)\rA_{C^{1+\eps}_*}\rp\lA(V,B)\rA_{H^s}\right.\\
					      &+\lp1+\lA a\rA_{C^\mez_*}+\lA\nabla\eta\rA_{C^\mez_*}\rp\lp1+\lA\zeta\rA_{H^{s-\mez}}\rp\\
					      &\left.+\lp1+\lA\nabla\eta\rA_{C^\mez_*}\rp\lA\psi\rA_{H^{s+\mez}}\rb,
		\end{aligned}
	\end{equation*}
	and
	\begin{equation*}
		\begin{aligned}
			\lA F_2\rA_{L^2}\leq K_0&\lb\lp1+\lA\nabla\eta\rA_{C^\mez_*}\rp\lA(V,B)\rA_{H^s}\right.\\
					      &+\lp1+\lA(V,B)\rA_{C^{1+\eps}_*}+\lA\partial_ta+V\cdot\nabla a\rA_{L^\infty}\rp\lp1+\lA\zeta\rA_{H^{s-\mez}}\rp\\
					      &+\lA\psi\rA_{H^{s+\mez}}\\
					      &\left.+\lp1+\lA a\rA_{C^\mez_*}+\lA\nabla\eta\rA_{C^\mez_*}\rp\lA U_s\rA_{L^2}\rb.
		\end{aligned}
	\end{equation*}
\end{prop}

\section{Estimates of the original unknowns} \label{sec:OrUnk}

In order to obtain a closed system of energy estimates, we need a control of the norms of~$\eta,\psi,V,B$, in terms of lower order norms and of norms of~$U_s,\theta_s$.
The formers will then be studied using transport equations on the various unknowns, and for the latters we will use the paralinearized system of Proposition~\ref{prop:parasyshold}.
\begin{prop} \label{prop:backhold}
	There holds
	\begin{equation*}
		\lA\eta\rA_{H^{s+\mez}}\leq K\lp\lA\theta_s\rA_{L^2}+\lA\zeta_s\rA_{H^{-1}}\rp,
	\end{equation*}
	\begin{multline*}
		\lA(V,B)\rA_{H^s}\leq K\lb1+\lA U_s\rA_{L^2}+\lp1+\lA\eta\rA_{C^\frac{3}{2}_*}\rp\lA(V,B)\rA_{H^{s-\mez}}\right.\\+
		\left.\lA(V,B)\rA_{C^{1+\eps}_*}\lp\lA\theta_s\rA_{L^2}+\lA\zeta_s\rA_{H^{-1}}\rp\rb,
	\end{multline*}
	and
	\begin{multline*}
		\lA\psi\rA_{H^{s+\mez}}\leq K\lb1+\lA U_s\rA_{L^2}+\lp1+\lA\eta\rA_{C^\frac{3}{2}_*}\rp\lA(V,B)\rA_{H^{s-\mez}}\right.\\+
		\left.\lA(V,B)\rA_{C^{1+\eps}_*}\lp\lA\theta_s\rA_{L^2}+\lA\zeta_s\rA_{H^{-1}}\rp+\lA\psi\rA_{L^2}\rb.
	\end{multline*}
\end{prop}
\begin{proof}
	 We start with the estimate on~$\eta$.
	 Fist we remark that 
	 \begin{align*}
	 	\lA\eta\rA_{H^{s+\mez}}&\leq\lA\eta\rA_{L^2}+\lA\nabla\eta\rA_{H^{s-\mez}}\\
				       &\leq K\lp1+\lA\zeta_s\rA_{H^{-\mez}}\rp,
	 \end{align*}
	 since $\zeta_s=\jap^s\nabla\eta$.
	 We then construct and use a parametrix to go back from~$\theta_s=T_q\zeta_s$ to~$\zeta_s$.
	 If our~$\eps$ is small enough, typically $0<\eps<s-d/2-1$, we choose~$N$ an integer such that $(N+1)\eps>1/2$, and we take $R=I-T_{1/q}T_q$, keeping in mind that $q=\sqrt{a/\lambda}$. Then
	 \begin{align*}
	 	\zeta_s&=T_{1/q}T_q\zeta_s+R\zeta_s\\
		       &=\lp I+R+\dots+R^N\rp T_{1/q}T_q\zeta_s+R^{N+1}\zeta_s.
	 \end{align*}
	 Then from the composition formula for paradifferential operators in Theorem~\ref{thm:para}, we have
	 \begin{align*}
	 	\lA R\rA_{H^\mu\rightarrow H^{\mu+\eps}}&\leq C\lp\M^{-\mez}_\eps(q)\M^\mez_0(1/q)+\M^{-\mez}_0(q)\M^\mez_\eps(1/q)\rp\\
							     &\leq K,
	 \end{align*}

	 and from the definition of~$q$, we see that
	 $$\lA T_{1/q}\rA_{L^2\rightarrow H^{-\mez}}\leq \M^\mez_0(1/q)\leq K.$$
	 Those estimations put together give 
	 $$\lA\eta\rA_{H^{s+\mez}}\leq K\lp\lA\theta_s\rA_{L^2}+\lA\zeta_s\rA_{H^{-1}}\rp.$$
	 
	 To simplify the equations, we worked with the unknown~$U_s=\jap^sV+T_\zeta\jap^sB$. We first show how to go back from this to~$U=V+T_\zeta B$.
	 We have
	 \begin{equation} \label{eq:UtoUs}
	 	\jap^s U=U_s+\lb\jap^s,T_\zeta\rb B,
	 \end{equation}

	 and Theorem~\ref{thm:para} gives
	 $$\lA\lb\jap^s,T_\zeta\rb B\rA_{L^2}\leq\lA\zeta\rA_{C^\mez_*}\lA B\rA_{H^{s-\mez}}.$$
	 Putting those two identities together gives
	 $$\lA U\rA_{H^s}\leq\lA U_s\rA_{L^2}+\lA\nabla\eta\rA_{C^\mez_*}\lA B\rA_{H^{s-\mez}}.$$
	 Then to get back to~$B$ from this, we take the divergence of~$U$ and use Proposition~\ref{prop:linkVB} to link~$V$ and $B$ and the paralinearization of the Dirichlet-Neumann~(\ref{eq:paraDN}), so that
	 \begin{align*}
	 	\div U&=\div T_\zeta B\\
		      &=G(\eta)B+\gamma'+T_{\div\zeta}B+T_\zeta\cdot\nabla B\\
		      &=T_p B+R(\eta)B+T_{\div\zeta}B+\gamma',
	 \end{align*}
	 where
	 $$p:=-\lambda+i\zeta\cdot\xi.$$
	 Now~$p$ is a symbol of order~$1$ and~$1/p$ of order~$-1$, with
	 \begin{equation} \label{eq:normp}
	 	\M^1_r(p)+\M^{-1}_r(1/p)\leq K\lA\nabla\eta\rA_{C^r_*}.
	 \end{equation}
	 Now we use a new parametrix from~$T_q B$ to~$B$, giving
	 \begin{equation} \label{eq:UtoB}
	 \begin{aligned}
	 	B&=T_{1/p}T_pB+(I-T_{1/p}T_p)B\\
		 &=T_{1/p}\div U-T_{1/p}\gamma'+T_{1/p}\lp-T_{\div\zeta}-R(\eta)\rp B+(I-T_{1/p}T_p)B. 
	 \end{aligned}
	 \end{equation}
	 This gives, using~(\ref{eq:normp}), Proposition~\ref{prop:linkVB}, Proposition~\ref{prop:remDN}, and the maximum principle~\ref{prop:maxprincder},
	 \begin{align*}
	 	\lA B\rA_{H^s}&\leq K\lp\lA U\rA_{H^s}+\lA\gamma'\rA_{H^{s-1}}+\lA\nabla\eta\rA_{C^{\mez}_*}\lA B\rA_{H^{s-\mez}}+\lA R(\eta)B\rA_{H^{s-1}}\rp\\
			      &\leq K\lp\lA U_s\rA_{L^2}+\lp1+\lA\eta\rA_{C^\frac{3}{2}_*}\rp\lA(V,B)\rA_{H^{s-\mez}}+\lA(V,B)\rA_{C^{1+\eps}_*}\lA\eta\rA_{H^{s+\mez}}\rp,
	 \end{align*}
	 which combined with the previous estimate on~$\eta$ gives the expected result.
	 Using the relation
	 $$U=V+T_\zeta B$$
	 gives the same estimation on~$V$.
	 
	 At last, we have the identity
	 $$\nabla\psi=V+B\nabla\eta,$$
	 and the quantities in the right side have all been estimated, so that a tame estimate on~$B\nabla\eta$ concludes the proof.
\end{proof}

For the version with a reference Sobolev norm, we have a simpler proposition.
\begin{prop}
	There holds
	\begin{equation*}
		\lA\eta\rA_{H^{s+\mez}}\leq K_0\lp\lA\theta_s\rA_{L^2}+\lA\zeta_s\rA_{H^{-1}}\rp,
	\end{equation*}
	\begin{equation*}
		\lA(V,B)\rA_{H^s}\leq K_0\lb\lA U_s\rA_{L^2}+\lA(V,B)\rA_{H^{s-\mez}}\rb,
	\end{equation*}
	and
	\begin{equation*}
		\lA\psi\rA_{H^{s+\mez}}\leq K_0\lb\lA U_s\rA_{L^2}+\lA(V,B)\rA_{H^{s-\mez}}\rb.
	\end{equation*}
\end{prop}

\begin{proof}
	The first estimate is a simple consequence of the previous proposition and of our hypothesis that~$K\leq K_0$.
	
	For the second one, we combine~(\ref{eq:UtoUs}) and~(\ref{eq:UtoB}) to get
	\begin{align*}
		B&=T_{\frac{1}{p}}\lb\div\jap^{-s}U_s+\gamma'\rb+\lb T_{\frac{1}{p}}\lp\div\jap^{-s}\lb\jap^s,T_\zeta\rb-R(\eta)-T_{\div\zeta}\rp+\lp I-T_{\frac{1}{p}}T_p\rp\rb B,\\
		 &:=M+RB.
	\end{align*}
	As before, this gives
	\begin{equation*}
		B=\lp I+R+\dots+R^N\rp M+R^{N+1}B,
	\end{equation*}
	where again~$\lp N+1\rp\eps>1/2$ with~$0<\eps<s_0-1/2-d/2$.
	Then we see using Proposition~\ref{prop:linkVB} that
	\begin{equation*}
		\lA M\rA_{H^s}\leq K_0\lb\lA U_s\rA_{L^2}+\lA(V,B)\rA_{H^{s-\mez}}\rb.
	\end{equation*}
	We see using~(\ref{eq:normp}) and Proposition~\ref{prop:remDN} that~$R$ is of order~$-\eps$, with
	\begin{equation*}
		\lA R\rA_{H^\sigma\rightarrow H^{\sigma+\eps}}\leq K_0
	\end{equation*}
	when~$s-\mez\leq\sigma\leq s$.
	This gives the estimate on~$B$, and the estimates on~$V$ and~$\psi$ are deduced from it as in the previous proposition.
\end{proof}

\section{Energy estimates} \label{sec:EneEst}
We start with a standard energy estimate on the now symmetric quasilinear system~(\ref{eq:syssym}).
\begin{prop} \label{prop:quases}
 The following estimates hold punctually in time
 \begin{equation} \label{eq:Uthetaesthold}
 \begin{aligned}
 	\tder\lA\lp U_s,\theta_s\rp\rA_{L^2}\leq&K\lp\lA a\rA_{C^\mez_*}+\lA\lp\partial_t+V\cdot\nabla\rp a\rA_{L^\infty}+Q\lp\lA\nabla\eta\rA_{C^\mez_*},\lA\lp V,B\rp\rA_{C^{1+\eps}_*}\rp\rp\\
						 &\times\lp\lA\lp U_s,\theta_s\rp\rA_{L^2}+\lA\psi\rA_{L^2}+\lA\zeta_s\rA_{H^{-1}}+\lA\lp V,B\rp\rA_{H^{s-\mez}}\rp,
 \end{aligned}
 \end{equation}
  where $Q$ is an explicit polynomial of degree~$3$,
  and
  \begin{equation} \label{eq:Uthetaestsob}
 \begin{aligned}
 	\tder\lA\lp U_s,\theta_s\rp\rA_{L^2}\leq&K_0\lp\lA a\rA_{C^\mez_*}+\lA\lp\partial_t+V\cdot\nabla\rp a\rA_{L^\infty}+\lA\nabla\eta\rA_{C^\mez_*}+\lA\lp V,B\rp\rA_{C^{1+\eps}_*}\rp\\
						 &\times\lp\lA\lp U_s,\theta_s\rp\rA_{L^2}+\lA\psi\rA_{L^2}+\lA\zeta_s\rA_{H^{-1}}+\lA\lp V,B\rp\rA_{H^{s-\mez}}\rp.
 \end{aligned}
 \end{equation}
\end{prop}

\begin{proof}
  Multiplication of the equations by~$U_s$ and~$\theta_s$ respectively, followed by integration in space gives
  $$\tder\lb\lA U_s\rA_{L^2}+\lA\theta_s\rA_{L^2}\rb\leq A+B+C,$$
  with
  \begin{align*}
  	A:=&\ls T_V\cdot\nabla U_s,U_s\rs+\ls T_V\cdot\nabla\theta_s,\theta_s\rs,\\
  	B:=&\ls T_\gamma \theta_s,U_s\rs-\ls T_\gamma U_s,\theta_s\rs,\\
  	C:=&\ls F_1,U_s\rs+\ls F_2,\theta_s\rs.
  \end{align*}
Now using Theorem~\ref{thm:para}, we see that 
$$\lA\lp T_V\cdot\nabla\rp^*+\lp T_V\cdot\nabla\rp\rA_{L^2\rightarrow L^2}\leq C\lA V\rA_{W^{1,\infty}},$$
and
$$\lA T_\gamma-T_\gamma^*\rA_{L^2\rightarrow L^2}\leq C\M^\mez_\mez(\gamma),$$
and the estimates on~$F_1$ and~$F_2$ of Proposition~\ref{prop:parasyshold} and Proposition~\ref{prop:parasyssob} complete the proof.
\end{proof}

The next proposition exploits the transport equations available on the remaining variables to close the system of estimates.
\begin{prop} \label{prop:tres}
	With the same~$Q$ as in the previous proposition, there holds, for~$A=\lA\psi\rA_{L^2}$, $A=\lA\zeta_s\rA_{H^{-1}}$ or~$A=\lA\lp V,B\rp\rA_{H^{s-\mez}}$,
	\begin{equation*}
	\begin{aligned}
		\tder A\leq&K\lp\lA a\rA_{C^\mez_*}+\lA\lp\partial_t+V\cdot\nabla\rp a\rA_{L^\infty}+Q\lp\lA\nabla\eta\rA_{C^\mez_*},\lA\lp V,B\rp\rA_{C^{1+\eps}_*}\rp\rp\\
						 &\times\lp\lA\lp U_s,\theta_s\rp\rA_{L^2}+\lA\psi\rA_{L^2}+\lA\zeta_s\rA_{H^{-1}}+\lA\lp V,B\rp\rA_{H^{s-\mez}}\rp,
	\end{aligned}
	\end{equation*}
	and
	\begin{equation*}
	\begin{aligned}
		\tder A\leq&K_0\lp\lA a\rA_{C^\mez_*}+\lA\lp\partial_t+V\cdot\nabla\rp a\rA_{L^\infty}+\lA\nabla\eta\rA_{C^\mez_*}+\lA\lp V,B\rp\rA_{C^{1+\eps}_*}\rp\\
						 &\times\lp\lA\lp U_s,\theta_s\rp\rA_{L^2}+\lA\psi\rA_{L^2}+\lA\zeta_s\rA_{H^{-1}}+\lA\lp V,B\rp\rA_{H^{s-\mez}}\rp.
	\end{aligned}
	\end{equation*}
\end{prop}
We will need the following lemma on transport equations
\begin{lem}
If~$\sigma>0$, and if $u$ solves
$$\lp\partial_t+V\cdot\nabla\rp u=f,$$
Then 
$$\tder\lA u\rA_{L^2}\lesssim \lA V\rA_{W^{1,\infty}}\lA u\rA_{L^2}+\lA f\rA_{L^2},$$
and
$$\tder\lA u\rA_{H^\sigma}\lesssim \lA V\rA_{W^{1,\infty}}\lA u\rA_{H^\sigma}+\lA\lp V-T_V\rp\cdot\nabla u\rA_{H^\sigma}+\lA\lp V-T_V\rp\cdot\jap^\sigma\nabla u\rA_{L^2}+\lA f\rA_{H^\sigma}.$$
\end{lem}

\begin{proof}
The~$L^2$ energy estimate is standard, and the Sobolev estimate follows from commuting the equation with~$\jap^\sigma$, using the~$L^2$ estimate, and observing that
\begin{align*}
 \lA\lb\jap^\sigma,V\rb\nabla u\rA_{L^2}\leq&\lA\lb\jap^\sigma,T_V\rb\cdot\nabla u\rA_{L^2}+\lA\lb\jap^\sigma,V-T_V\rb\cdot\nabla u\rA_{L^2}\\
				   \leq&\lA V\rA_{W^{1,\infty}}\lA u\rA_{H^\sigma}+\lA\lp V-T_V\rp\cdot\nabla u\rA_{H^\sigma}+\lA\lp V-T_V\rp\cdot\jap^\sigma\nabla u\rA_{L^2}.
\end{align*}
\end{proof}

\begin{proof}[Proof of Proposition \ref{prop:tres}]
	First, from the equation on~$\psi$ and the definitions of~$V$ and~$B$ we have the transport equation
	$$\lp\partial_t+V\cdot\nabla\rp\psi=-g\eta+\mez V^2+\mez B^2.$$
	The previous~$L^2$ estimate and a simple tame estimate on the~$L^2$ norms of~$V^2$ and~$B^2$ give the estimate on~$\psi$.
	
	We then recall equation~(\ref{eq:zetaeq}),
	$$\left(\partial_t+V\cdot\nabla\right)\zeta=G(\eta)V+\zeta G(\eta)B+\gamma,$$
	and use the previous Sobolev estimate with $\sigma=s-1$ to get
	\begin{align*}
		\tder\lA\zeta\rA_{H^{s-1}}\lesssim&\lA V\rA_{W^{1,\infty}}\lA\zeta\rA_{H^{s-1}}+\lA\lp V-T_V\rp\cdot\nabla\zeta\rA_{H^{s-1}}+\lA\lp V-T_V\rp\cdot\jap^{s-1}\nabla\zeta\rA_{L^2}\\
		                                  &+\lA G(\eta)V\rA_{H^{s-1}}+\lA\zeta\rA_{L^\infty}\lA G(\eta)B\rA_{H^{s-1}}+\lA\zeta\rA_{H^{s-1}}\lA G(\eta)B\rA_{L^\infty}\\
		                                  &+\lA\gamma\rA_{H^{s-1}}.
	\end{align*}
	Using the parproduct rules from Theorem~\ref{thm:paraprod} gives
	\begin{align*}
		\lA\lp V-T_V\rp\cdot\nabla\zeta\rA_{H^{s-1}}&\leq\lA\nabla\zeta\rA_{C^{-\mez}_*}\lA V\rA_{H^{s-\mez}}\\
							&\leq\lA\nabla\eta\rA_{C^\mez_*}\lA V\rA_{H^{s-\mez}}
	\end{align*}
	and
	\begin{align*}
		\lA\lp V-T_V\rp\cdot\jap^{s-1}\nabla\zeta\rA_{L^2}&\leq\lA\jap^{s-1}\nabla\zeta\rA_{C^{1-s+\eps}_*}\lA V\rA_{H^s}\\
							      &\leq\lA\nabla\eta\rA_{C^\eps_*}\lA V\rA_{H^s}.
	\end{align*}
	To estimate the Dirichlet-Neumann operators, we use Proposition~\ref{prop:DN}, and~$\gamma$ is estimated using~(\ref{eq:rembothold}) or~(\ref{eq:rembotsob}).
	
	Recall also that $B$ follows equation~(\ref{eq:Beq}),
	$$\left(\partial_t+V\cdot\nabla\right)B=a-g.$$
	The Sobolev estimate gives
	\begin{align*}
		\tder\lA B\rA_{H^{s-\mez}}\lesssim&\lA V\rA_{W^{1,\infty}}\lA B\rA_{H^{s-\mez}}+\lA\lp V-T_V\rp\cdot\nabla B\rA_{H^{s-\mez}}\\
		                                  &+\lA\lp V-T_V\rp\cdot\jap^{s-\mez}\nabla B\rA_{L^2}+\lA a-g\rA_{H^{s-\mez}},
	\end{align*}
	and as for~$\zeta$, we have
	\begin{align*}
		\lA\lp V-T_V\rp\cdot\nabla B\rA_{H^{s-\mez}}&\leq\lA\nabla B\rA_{L^\infty}\lA V\rA_{H^{s-\mez}}
	\end{align*}
	and
	\begin{align*}
		\lA\lp V-T_V\rp\cdot\jap^{s-\mez}\nabla B\rA_{L^2}&\leq\lA\jap^{s-\mez}\nabla B\rA_{C^{\mez-s+\eps}_*}\lA V\rA_{H^{s-\mez}}\\
							      &\leq\lA B\rA_{C^{1+\eps}_*}\lA V\rA_{H^{s-\mez}}.
	\end{align*}
	$a-g$ is estimated using Proposition~\ref{prop:Tay}.
	
	At last~$V$ follows equation~(\ref{eq:Veq}),
	$$\left(\partial_t+V\cdot\nabla\right)V=-a\zeta,$$
	so we use the same bound for the commutator that we used for~$B$ and remark that
	$$\lA a\zeta\rA_{H^{s-\mez}}\lesssim \lA a-g\rA_{H^{s-\mez}}\lA\zeta\rA_{L^\infty}+\lp g+\lA a\rA_{L^\infty}\rp\lA\zeta\rA_{H^{s-\mez}}$$
	from which the proposition follows.
\end{proof}

\section{Proof of the main results} \label{sec:Proof}

The main theorems will follow as usual from the expression of the energy estimates. The one with only H\"{o}lder components is
\begin{prop}
Let $d\geq1$, $s>1+d/2$, and $\eps>0$. Let $(\eta,\psi)$ be a solution of the water-waves system~(\ref{eq:ZaCrSu}) on~$[0,T]$
from theorem~\ref{thm:AlazardBurqZuilyExw/ST}, and define $V,B,h,a,c$ in the same way. Then
\begin{multline} \label{eq:boundhold}
\sup_{0\leq t\leq T}\lA(\eta,\psi,V,B)(t)\rA_{H^{s+\frac{1}{2}}(\R^d)\times H^{s+\frac{1}{2}}(\R^d)\times H^s(\R^d)\times H^s(\R^d)}\\
\leq\mathcal{F}\left(T,\sup_{0\leq t\leq T}\frac{1}{h},\sup_{0\leq t\leq T}\frac{1}{c},
\sup_{0\leq t\leq T}\lA\eta(t)\rA_{W^{1+\eps,\infty}(\R^d)},
  \sup_{0\leq t\leq T}\lA(V,B)(t)\rA_{W^{\eps,\infty}(\R^d)},\right.\\
  \sup_{0\leq t\leq T}\lA a\rA_{W^{\eps,\infty}(\R^d)},
  \int_0^T\lA(\partial_t a+V\cdot\nabla a)(t)\rA_{L^\infty(\R^d)}\;\mathrm{d}t,
  \int_0^T\lA a(t)\rA_{W^{\frac{1}{2},\infty}(\R^d)}\;\mathrm{d}t,\\
  \left.\int_0^T\lA\nabla\eta(t)\rA^3_{W^{\frac{1}{2},\infty}(\R^d)}\;\mathrm{d}t,
  \int_0^T\lA(V,B)(t)\rA^3_{W^{1+\eps,\infty}(\R^d)}\;\mathrm{d}t\right),
\end{multline}
with $\mathcal{F}$ a positive, strictly increasing function of each of its variable, depending only on $d$, $s$, $\eps$,
the bottom $\Gamma$, and on $\lA\eta_0,\psi_0,V_0,B_0\rA_{H^{s+\frac{1}{2}}(\R^d)\times H^{s+\frac{1}{2}}(\R^d)\times H^s(\R^d)\times H^s(\R^d)}$.
\end{prop}

\begin{proof}
We see from propositions~\ref{prop:quases} and~\ref{prop:tres} that if
$$\mathcal{A}:=\lA\lp U_s,\theta_s\rp\rA_{L^2}+\lA\psi\rA_{L^2}+\lA\zeta_s\rA_{H^{-1}}+\lA\lp V,B\rp\rA_{H^{s-\mez}},$$
and if 
$$\mathcal{B}:=\lp\lA a\rA_{C^\mez_*}+\lA\lp\partial_t+V\cdot\nabla\rp a\rA_{L^\infty}+Q\lp\lA\nabla\eta\rA_{C^\mez_*},\lA\lp V,B\rp\rA_{C^{1+\eps}_*}\rp\rp,$$
with~$Q$ the polynomial of degree~$3$ of those propositions, then for all~$t\in\lb0,T\rb$,
\begin{equation*}
	\tder\mathcal{A}(t)\leq K(t)\times\mathcal{B}(t)\times\mathcal{A}(t)\leq\mathcal{K}\times\mathcal{B}(t)\times\mathcal{A}(t),
\end{equation*}
where
$$\mathcal{K}=\sup_{0\leq t\leq T}K(t).$$

Using Gr\"{o}nwall's lemma gives
$$\sup_{0\leq t\leq T}\mathcal{A}(t)\leq\mathcal{K}\mathcal{A}(0)\exp\lp\int_0^T\mathcal{B}(t)\d t\rp.$$
Using H\"{o}lder inequality to bound
$$\int_0^TQ\lp\lA\nabla\eta\rA_{C^\mez_*},\lA\lp V,B\rp\rA_{C^{1+\eps}_*}\rp\d t$$
gives 
\begin{multline}
	\sup_{0\leq t\leq T}\mathcal{A}(t)\leq
	  \mathcal{F}\left(T,\sup_{0\leq t\leq T}\frac{1}{h},\sup_{0\leq t\leq T}\frac{1}{c},
\sup_{0\leq t\leq T}\lA\eta(t)\rA_{W^{1+\eps,\infty}(\R^d)},
  \sup_{0\leq t\leq T}\lA(V,B)(t)\rA_{W^{\eps,\infty}(\R^d)},\right.\\
  \sup_{0\leq t\leq T}\lA a\rA_{W^{\eps,\infty}(\R^d)},
  \int_0^T\lA(\partial_t a+V\cdot\nabla a)(t)\rA_{L^\infty(\R^d)}\;\mathrm{d}t,
  \int_0^T\lA a(t)\rA_{W^{\frac{1}{2},\infty}(\R^d)}\;\mathrm{d}t,\\
  \left.\int_0^T\lA\nabla\eta(t)\rA^3_{W^{\frac{1}{2},\infty}(\R^d)}\;\mathrm{d}t,
  \int_0^T\lA(V,B)(t)\rA^3_{W^{1+\eps,\infty}(\R^d)}\;\mathrm{d}t\right),
\end{multline}
with $\mathcal{F}$ a positive, strictly increasing function of each of its variable, depending only on $d$, $s$, $\eps$,
the bottom $\Gamma$, and on~$\mathcal{A}(0)$, which is easily seen to be controlled by the initial data
$\lA\eta_0,\psi_0,V_0,B_0\rA_{H^{s+\frac{1}{2}}\times H^{s+\frac{1}{2}}\times H^s\times H^s}$.

The water-waves system~\ref{eq:ZaCrSu} is Hamiltonian, and the Hamiltonian~(\ref{eq:Ham}) controls the~$L^2$ norm of~$\eta$, so that
$$\sup_{0\leq t\leq T}\lA\eta(t)\rA_{L^2}\leq\lA\eta_0\rA_{H^{s+\mez}}$$.

Now to finish the proof of the estimate, we remark that for any~$\nu>0$ there exists a constant~$C_\nu>0$ such that
\begin{equation*}
	\lA\lp\eta,V,B\rp\rA_{C^{\frac{3}{2}}_*\times C^{1+\eps}_*\times C^{1+\eps}_*}\leq C_\nu\lA\lp\eta,V,B\rp\rA_{L^2}+\nu\lA\lp\eta,V,B\rp\rA_{H^{s+\mez}\times H^s\times H^s},
\end{equation*}
which, combined with Proposition~\ref{prop:backhold} gives that~$\lA\lp\eta,V,B\rp\rA_{H^{s+\mez}\times H^s\times H^s}$ is finite as soon as the right side of~(\ref{eq:boundhold}) is bounded.
\end{proof}

The energy estimate with reference Sobolev index is proved along the same lines, using the corresponding estimates.
\begin{prop}
Let $d\geq1$, $s>1+d/2$, $s>s_0>1/2+d/2$ and $s_0-1/2-d/2>\eps>0$. Let $(\eta,\psi)$ be a solution of the water-waves system~(\ref{eq:ZaCrSu}) on~$[0,T]$
from theorem~\ref{thm:AlazardBurqZuilyExw/ST}, and define $V,B,h,a,c$ in the same way. Then
\begin{multline}
\sup_{0\leq t\leq T}\lA(\eta,\psi,V,B)(t)\rA_{H^{s+\frac{1}{2}}(\R^d)\times H^{s+\frac{1}{2}}(\R^d)\times H^s(\R^d)\times H^s(\R^d)}\\
\leq\mathcal{F}\left(T,\sup_{0\leq t\leq T}\frac{1}{h},\sup_{0\leq t\leq T}\frac{1}{c},
\sup_{0\leq t\leq T}\lA\lp\eta,\psi,V,B\rp(t)\rA_{H^{s_0+\mez}(\R^d)\times H^{s_0+\mez}(\R^d)\times H^{s_0}(\R^d)\times H^{s_0}(\R^d)},\right.\\
  \sup_{0\leq t\leq T}\lA a\rA_{W^{\eps,\infty}(\R^d)},
  \int_0^T\lA(\partial_t a+V\cdot\nabla a)(t)\rA_{L^\infty(\R^d)}\;\mathrm{d}t,
  \int_0^T\lA a(t)\rA_{W^{\frac{1}{2},\infty}(\R^d)}\;\mathrm{d}t,\\
  \left.\int_0^T\lA\nabla\eta(t)\rA_{W^{\frac{1}{2},\infty}(\R^d)}\;\mathrm{d}t,
  \int_0^T\lA(V,B)(t)\rA_{W^{1+\eps,\infty}(\R^d)}\;\mathrm{d}t\right),
\end{multline}
with $\mathcal{F}$ a positive, strictly increasing function of each of its variable, depending only on $d$, $s$, $s_0$, $\eps$,
the bottom $\Gamma$, and on $\lA\eta_0,\psi_0,V_0,B_0\rA_{H^{s+\frac{1}{2}}(\R^d)\times H^{s+\frac{1}{2}}(\R^d)\times H^s(\R^d)\times H^s(\R^d)}$.
\end{prop}

At last Theorem~\ref{thm:mainsob2} is a consequence of the following proposition from~\cite{AlazardBurqZuilyStrEstGrav},
\begin{prop}[Proposition~3.6 of~\cite{AlazardBurqZuilyStrEstGrav}]
	For~$s_0>3/4+d/2$, and $0<\eps<s_0-3/4-d/2$,
	\begin{multline}
		\lA a\rA_{C^\mez_*}+\lA\lp\partial_t+V\cdot\nabla\rp a\rA_{L^\infty}\\\leq
		  \mathcal{F}\lp\lA\lp\eta,\psi,V,B\rp(t)\rA_{H^{s_0+\mez}\times H^{s_0+\mez}\times H^{s_0}\times H^{s_0}}\rp
		  \lb1+\lA\eta\rA_{C^{\mez+\eps}_*}+\lA\lp V,B\rp\rA_{C^{1+\eps}_*}\rb.
	\end{multline}
\end{prop}

\begin{proof}
	The estimate on~$\lA a\rA_{C^\mez_*}$ is proved in details in~\cite{AlazardBurqZuilyStrEstGrav}. We will record here a proof of the estimate on~$\lA\lp\partial_t+V\cdot\nabla\rp a\rA_{L^\infty}$.
	First, observe that 
	\begin{align*}
		\lp\partial_t+V\cdot\nabla\rp a&=\left.\lp\partial_t+\nab\phi\cdot\nab\rp\lp-\partial_yP\rp\ra_{y=\eta}\\
						&=-\left.\partial_y\lp\partial_t+\nab\phi\cdot\nab\rp P\ra_{y=\eta}-\left.\lb\lp\partial_t+\nab\phi\cdot\nab\rp,\partial_y\rb P\ra_{y=\eta}.
	\end{align*}
	The second term on the right hand side is
	\begin{align*}
		-\left.\lb\lp\partial_t+\nab\phi\cdot\nab\rp,\partial_y\rb P\ra_{y=\eta}&=\nabla_{\!x}\partial_y\phi\cdot\nabla_{\!x}P\rvert_{y=\eta}+\partial^2_y\phi\partial_yP\rvert_{y=\eta}\\
											 &=a\nabla\eta\cdot\nabla_{\!x}\partial_y\phi\rvert_{y=\eta}+a\Delta_x\phi\rvert_{y=\eta}\\
											 &=a\div V,
	\end{align*}
	where we have used that since~$P\rvert_{y=\eta}=0$, 
	\begin{equation*}
		0=\nabla(P\rvert_{y=\eta})=(\nabla_{\!x}P)\rvert_{y=\eta}+\nabla\eta(\partial_yP)\rvert_{y=\eta},
	\end{equation*}
	and that
	\begin{equation*}
		\div V=\Delta_x\phi\rvert_{y=\eta}+\nabla\eta\cdot(\nabla_{\!x}\partial_y\phi)\rvert_{y=\eta}.
	\end{equation*}
	The proposition will then be proved once we have shown that 
	\begin{equation} \label{eq:convPell}
		\lA\partial_y\lp\partial_t+\nab\phi\cdot\nab\rp P\rA_{X^{s_0-\frac{3}{4}}}\leq\mathcal{F}\lp\lA\lp\eta,\psi,V,B\rp(t)\rA_{H^{s_0+\mez}\times H^{s_0+\mez}\times H^{s_0}\times H^{s_0}}\rp.
	\end{equation}
	This is a consequence of the following elliptic equation
	\begin{align*}
		\Delta_{x,y}\lp\partial_t+\nab\phi\cdot\nab\rp P&=\lp\partial_t+\nab\phi\cdot\nab\rp\Delta_{x,y}P+\lb\Delta_{x,y},\lp\partial_t+\nab\phi\cdot\nab\rp\rb P\\
								 &=-\lp\partial_t+\nab\phi\cdot\nab\rp\la\nab^2\phi\ra^2+2\nab^2\phi\cdot\nab^2P\\
								 &=-2\nab^2\phi\cdot\nab^2(-\partial_t\phi-\mez\la\nab\phi\ra^2-gy)+2\nab^2\phi\cdot\nab^2P\\
								 &=4\nab^2\phi\cdot\nab^2P.
	\end{align*}
	Now~(\ref{eq:convPell}) follow along the same lines as in Proposition~\ref{prop:Tay}, using the regularity on~$P$ already established in this proposition.
\end{proof}

\begin{appendices}

\section{Paradifferential calculus} \label{ap:para}
We review the fundamental results of Bony's paradifferential calculus, introduced in~\cite{BonyPara}, following M\'etivier presentation in \cite{MetivierParadiff}. 
\subsection{Paradifferential operators}
\begin{defn}
Given~$\rho\in [0, 1]$ and~$m\in\R$,~$\Gamma_{\rho}^{m}(\R^d)$ denotes the space of
locally bounded functions~$a(x,\xi)$
on~$\R^d\times(\R^d\setminus 0)$,
which are~$C^\infty$ with respect to~$\xi$ for~$\xi\neq0$ and
such that, for all~$\alpha\in\N^d$ and all~$\xi\neq 0$, the function
$x\mapsto \partial_\xi^\alpha a(x,\xi)$ belongs to~$W^{\rho,\infty}(\R^d)$ and there exists a constant
$C_\alpha$ such that,
\begin{equation*}
\forall\lvert\xi\rvert\geq\mez,\quad 
\lA\partial_\xi^\alpha a(\cdot,\xi)\rA_{W^{\rho,\infty}(\R^d)}\leq C_\alpha
(1+\lvert\xi\rvert)^{m-\lvert\alpha\rvert}.
\end{equation*}
\end{defn}

From a symbol~$a$, we define
the paradifferential operator~$T_a$ by
\begin{equation*}
\widehat{T_a u}(\xi)=(2\pi)^{-d}\int \chi(\xi-\eta,\eta)\widehat{a}(\xi-\eta,\eta)\psi(\eta)\widehat{u}(\eta)
\d\eta,
\end{equation*}
where
$\widehat{a}(\theta,\xi)=\int e^{-ix\cdot\theta}a(x,\xi)\d x$
is the Fourier transform of~$a$ with respect to the first variable. 
$\chi$ and~$\psi$ are two fixed~$C^\infty$ functions such that:
\begin{equation*}
\psi(\eta)=0\quad\text{for } \la\eta\ra\leq1,\qquad
\psi(\eta)=1\quad\text{for }\la\eta\ra\geq2,
\end{equation*}
and~$\chi(\theta,\eta)$ 
satisfies, for~$0<\eps_1<\eps_2$ small,
$$
\chi(\theta,\eta)=1\quad\text{if}\quad\la\theta\ra\leq\eps_1\la\eta\ra,\qquad
\chi(\theta,\eta)=0\quad\text{if}\quad\la\theta\ra\geq\eps_2\la\eta\ra,
$$
and such that
$$
\forall(\theta,\eta)\,:\qquad \la\partial_\theta^\alpha\partial_\eta^\beta\chi(\theta,\eta)\ra\leq 
C_{\alpha,\beta}(1+\la\eta\ra)^{-\la\alpha\ra-\la\beta\ra}.
$$

\subsection{Symbolic calculus}
In order to get quantitative results about operator norm estimates, we introduce the following semi-norms
\begin{defn}
For~$m\in\R$,~$\rho\in [0,1]$ and~$a\in \Gamma^m_{\rho}(\R^d)$, put
\begin{equation*}
M_{\rho}^{m}(a)= 
\sup_{\lvert\alpha\rvert\leq\frac{3d}{2}+1+\rho ~}\sup_{\lvert\xi\rvert\geq1/2~}
\lA (1+\lvert\xi\rvert)^{\lvert\alpha\rvert-m}\partial_\xi^\alpha a(\cdot,\xi)\rA_{W^{\rho,\infty}(\R^d)}.
\end{equation*}
\end{defn}

The natural spaces for paradifferential operators to act on are the following
\begin{defn}[Besov spaces]
Consider a dyadic decomposition of the identity:
$I=\Delta_{-1}+\sum_{q=0}^{\infty}\Delta_q$. If $s$ is any real number, we define the Besov class $B^s_{p,q}(\R^d)$ as the space of 
functions of tempered distributions $u$ such that
$$\lA u\rA_{B^s_{p,q}}:=\left(\sum_{k\in\N}2^{pks}\lA\Delta_ku\rA_{L^q}^p\right)^\frac{1}{p}<+\infty.$$
\end{defn}
We rename the space $B^s_{\infty,\infty}$ as the Zygmund space $C^s_*$.
\begin{rem}
The space $C^s_*(\R^d)$ is the usual H\"{o}lder space $W^{s,\infty}(\R^d)$ if $s>0$ is not an integer.
The space $B^s_{2,2}(\R^d)$ is the usual Sobolev space $H^s(\R^d)$.
\end{rem}

We will mainly use the Sobolev and Zygmund cases. 
\begin{defn}
Let~$m\in\R$.
An operator~$T$ is said to be of  order~$m$ if, for all~$\mu\in\R$,
it is bounded from~$H^{\mu}$ to~$H^{\mu-m}$ and from~$C^{\mu}_*$ to~$C^{\mu-m}_*$, where the~$C^\mu_*$ are the Zygmund spaces defined in  
\end{defn}

We resume most of the calculus properties we will use in the following theorem.
\begin{thm} \label{thm:para}
Let~$m\in\R$ and~$\rho\in [0,1]$. 

$(i)$ If~$a \in \Gamma^m_0(\R^d)$, then~$T_a$ is of order~$m$. 
Moreover, for all~$\mu\in\R$ there exists a constant~$K$ such that
\begin{equation*}
\lA T_a \rA_{H^{\mu}\rightarrow H^{\mu-m}}\leq K M_{0}^{m}(a),\qquad  
\lA T_a \rA_{C^{\mu}_*\rightarrow C^{\mu-m}_*}\leq K M_{0}^{m}(a).
\end{equation*}
$(ii)$ If~$a\in \Gamma^{m}_{\rho}(\R^d), b\in \Gamma^{m'}_{\rho}(\R^d)$ then 
$T_a T_b -T_{a b}$ is of order~$m+m'-\rho$. 
Moreover, for all~$\mu\in\R$ there exists a constant~$K$ such that
\begin{equation*}
\begin{aligned}
\lA T_a T_b  - T_{a b}   \rA_{H^{\mu}\rightarrow H^{\mu-m-m'+\rho}}&\leq 
K M_{\rho}^{m}(a)M_{0}^{m'}(b)+K M_{0}^{m}(a)M_{\rho}^{m'}(b),\\
\lA T_a T_b  - T_{a b}   \rA_{C^{\mu}_*\rightarrow C^{\mu-m-m'+\rho}_*}&\leq 
K M_{\rho}^{m}(a)M_{0}^{m'}(b)+K M_{0}^{m}(a)M_{\rho}^{m'}(b).
\end{aligned}
\end{equation*}
$(iii)$ Let~$a\in \Gamma^{m}_{\rho}(\R^d)$. Denote by 
$(T_a)^*$ the adjoint operator of~$T_a$ and by~$\overline{a}$ the complex conjugate of~$a$. Then 
$(T_a)^* -T_{\overline{a}}$ is of order~$m-\rho$. 
Moreover, for all~$\mu$ there exists a constant~$K$ such that
\begin{equation*}
\lA (T_a)^*   - T_{\overline{a}}   \rA_{H^{\mu}\rightarrow H^{\mu-m+\rho}}\leq 
K M_{\rho}^{m}(a), \qquad 
\lA (T_a)^*   - T_{\overline{a}}   \rA_{C^{\mu}_*\rightarrow C^{\mu-m+\rho}_*}\leq
K M_{\rho}^{m}(a).
\end{equation*}
\end{thm}

In this article, we need to consider paradifferential operators with negative regularity. 
As a consequence, we extend our previous definition.
\begin{defn}
For~$m\in \R$ and~$\rho\in (-\infty, 0)$,~$\Gamma^m_\rho(\R^d)$ denotes the space of
distributions~$a(x,\xi)$
on~$\R^d\times(\R^d\setminus 0)$,
which are~$C^\infty$ with respect to~$\xi$ and 
such that, for all~$\alpha\in\N^d$ and all~$\xi\neq 0$, the function
$x\mapsto \partial_\xi^\alpha a(x,\xi)$ belongs to $C^\rho_*(\R^d)$ and there exists a constant
$C_\alpha$ such that
\begin{equation*}
\forall\lvert\xi\rvert\geq\mez,\quad \lA\partial_\xi^\alpha a(\cdot,\xi)\rA_{C^\rho_*}\leq C_\alpha
(1+\lvert\xi\rvert)^{m-\lvert\alpha\rvert}.
\end{equation*}
Then~$\dot\Gamma_{\rho}^{m}(\R^d)$ denotes the subspace of 
$\Gamma_{\rho}^{m}(\R^d)$ which consists of symbols 
$a(x,\xi)$ which are homogeneous of degree~$m$  
with respect to~$\xi$. For~$a\in \Gamma^m_\rho$, we define 
\begin{equation*}
M_{\rho}^{m}(a)= 
\sup_{\lvert\alpha\rvert\leq\frac {3d} 2 +\rho+1 ~}\sup_{\lvert\xi\rvert \geq 1/2~}
\lA (1+\lvert\xi\rvert)^{\lvert\alpha\rvert-m}\partial_\xi^\alpha a(\cdot,\xi)\rA_{C^{\rho}_*(\R^d)}.
\end{equation*}
\end{defn}

\subsection{Paraproducts and product rules}
If~$a=a(x)$ is a function of~$x$ only, the paradifferential operator~$T_a$ is called a paraproduct. 
\begin{defn}
Given two functions~$a,b$ defined on~$\R^d$ we define the remainder 
$$
R(a,u)=au-T_a u-T_u a.
$$
\end{defn}
We record here various estimates about paraproducts (see chapter 2 in~\cite{BahouriCheminDanchinFourNL} or~\cite{CheminPerIncFlu}).
\begin{thm} \label{thm:paraprod}
\begin{enumerate}
\item  Let~$\alpha,\beta\in \R$. If~$\alpha+\beta>0$ then
\begin{align}
&\lA R(a,u) \rA _{H^{\alpha + \beta-\frac{d}{2}}(\R^d)}
\leq K \lA a \rA _{H^{\alpha}(\R^d)}\lA u\rA _{H^{\beta}(\R^d)}, \label{eq:pararem1}\\ 
&\lA R(a,u) \rA _{C^{\alpha + \beta}_*(\R^d)} 
\leq K \lA a \rA _{C^{\alpha}_*(\R^d)}\lA u\rA _{C^{\beta}_*(\R^d)},\\
&\lA R(a,u) \rA _{H^{\alpha + \beta}(\R^d)} \leq K \lA a \rA _{C^{\alpha}_*(\R^d)}\lA u\rA _{H^{\beta}(\R^d)}.\label{eq:paradifremsob}
\end{align}

\item   Let~$m>0$ and~$s\in \R$. Then
\begin{align}
&\lA T_a u\rA_{H^s}\leq K \lA a\rA_{L^\infty}\lA u\rA_{H^s} \label{eq:paraprod1}\\
&\lA T_a u\rA_{H^{s-m}}\leq K \lA a\rA_{C^{-m}_*}\lA u\rA_{H^{s}},\label{eq:paraprod2}\\
&\lA T_a u\rA_{C^{s-m}_*}\leq K \lA a\rA_{C^{-m}_*}\lA u\rA_{C^{s}_*},\\
&\lA T_a u\rA_{C^{s}_*}\leq K \lA a\rA_{L^\infty}\lA u\rA_{C^{s}_*}.
\end{align}
\item Let~$s_0,s_1,s_2$ be such that 
$s_0\le s_2$ and~$s_0 < s_1 +s_2 -\frac{d}{2}$, 
then
\begin{equation*}
\lA T_a u\rA_{H^{s_0}}\le K \lA a\rA_{H^{s_1}}\lA u\rA_{H^{s_2}}.
\end{equation*}
\end{enumerate}
\end{thm}
We shall mainly use the following consequences
\begin{prop}
Let~$\mu,m,n\in\R$, $\mu,m,n>0$ and $m,n\not\in\N$. Then
\begin{equation} \label{eq:tamestsob}
  \lA u_1 u_2\rA_{H^\mu}\leq K \bigl(\lA u_1\rA_{C^{-n}_*}\lA u_2\rA_{H^{\mu+n}}+\lA u_2\rA_{C^{-m}_*}\lA u_1\rA_{H^{\mu+m}}\bigr),
\end{equation}
and
\begin{equation}
\lA u_1 u_2\rA_{C^\mu_*}\leq K \bigl(\lA u_1\rA_{C^{-n}_*}\lA u_2\rA_{C^{\mu+n}_*}+\lA u_2\rA_{C^{-m}_*}\lA u_1\rA_{C^{\mu+m}_*}\bigr).
\end{equation}
\end{prop}
Recall from \cite{AlazardBurqZuilyExw/ST} the following 
\begin{prop}
\begin{enumerate}
  \item  
  Let~$s>d/2$ and consider~$F\in C^\infty(\C^N)$ such that~$F(0)=0$. 
Then there exists a non-decreasing function~$\mathcal{F}\colon\R^+\rightarrow\R^+$ 
such that
\begin{equation*}
\lA F(U)\rA_{H^s}\leq\mathcal{F}\bigl(\lA U\rA_{L^\infty}\bigr)\lA U\rA_{H^s},
\end{equation*}
for any~$U\in H^s(\R^d)^N$. 
\item 
Let~$s\geq 0$ and consider~$F\in C^\infty(\C^N)$ such that~$F(0)=0$. 
Then there exists a non-decreasing function~$\mathcal{F}\colon\R^+\rightarrow\R^+$ such that
\begin{equation*}
\lA F(U)\rA_{C^{s}_*}\leq\mathcal{F}\bigl(\lA U\rA_{L^\infty}\bigr)\lA U\rA_{C^{s}_*},
\end{equation*}
for any~$U\in C^{s}_*(\R^d)^N$. 
\end{enumerate}
\end{prop}
We also recall from \cite{AlazardBurqZuilyExw/ST} the following generalization of (\ref{eq:paraprod2}).
\begin{prop}
Let~$\rho<0$,~$m\in \R$ and~$a\in \dot{ \Gamma}^m_\rho$. Then the operator~$T_a$ is of order~$m-\rho$:
\begin{equation*}
\begin{aligned}
\la T_a \ra_{H^s \rightarrow H^{s-(m- \rho)}}&\leq C M_{\rho}^{m}(a),\\
\la T_a \ra_{C^s_* \rightarrow C^{s-(m- \rho)}_*}&\leq C M_{\rho}^{m}(a).
\end{aligned}
\end{equation*}
\end{prop}

We will also need the following commutator estimate between a paradifferential operator and a convective derivative
\begin{lem}[Lemma 2.16 of~\cite{AlazardBurqZuilyExw/ST}] \label{lem:comconv}
	Let~$V\in C^0([0,T];\C^{1+\eps}_*(\R^d))$ and let~$p=p(t,x,\xi)$ be a symbol homogeneous of order~$m\in\R$ in~$\xi$.
	Then there exists~$C>0$ independent if~$p$ and~$V$, such that for any~$t\in[0,T]$ and any ~$u\in C^0([0,T;H^m(\R^d))$,
	\begin{multline*} 
		\lA\left[T_p,\partial_t+T_V\cdot\nabla\right]u(t)\rA_{L^2(\R^d)}\\
		\leq C\left\lbrace\mathcal{M}^m_0(p)\lA V(t)\rA_{C^{1+\eps}_*}+\mathcal{M}^m_0\lp\partial_tp+V\cdot\nabla p\rp\right\rbrace\lA u(t)\rA_{H^m(\R^d)}.
	\end{multline*}
\end{lem}

\subsection{Parabolic evolution equation}
Given~$I\subset\R$, let~$\Gamma^m_\rho(I\times\R^d)$ be the set of symbols $a(z,x,\xi)$ satisfying
\begin{equation*}
\mathcal{M}^m_\rho(a)=\sup_{z\in I}
\sup_{\la\alpha\ra\le \frac{3d}{2}+\rho+1 ~}\sup_{\la\xi\ra \ge 1/2~}
\lA (1+\la\xi\ra)^{\la\alpha\ra-m}\partial_\xi^\alpha a(z;\cdot,\xi)\rA_{W^{\rho,\infty}(\R^d)}<+\infty.
\end{equation*}
We need to study the parabolic evolution equation 
\begin{equation*}
\partial_z w + T_p w =f,\quad w\arrowvert_{z=z_0}=w_0,
\end{equation*}
with an elliptic symbol $p\in\Gamma^m_\rho(I\times\R^d)$, which means that 
\begin{equation} \label{eq:ellipcnd}
\ p(z;x,\xi) \geq c \la\xi\ra,
\end{equation}
for some positive constant~$c$.

Define for $\mu\in\R$ the spaces
\begin{equation} \label{eq:parabspaces}
\begin{aligned}
  X^\mu(I)&=L^\infty_z(I;H^\mu(\R^d))\cap L^2_z(I;H^{\mu+\frac{1}{2}}(\R^d)),\\
  Y^\mu(I)&=L^1_z(I;H^\mu(\R^d))+L^2_z(I;H^{\mu-\frac{1}{2}}(\R^d)).
\end{aligned}
\end{equation}
We will use the following proposition from~\cite{AlazardBurqZuilyExw/ST}.
\begin{prop}[Proposition~2.18 of~\cite{AlazardBurqZuilyExw/ST}] \label{prop:parabsob}
Let~$\rho\in (0,1)$,~$J=[z_0,z_1]\subset\R$, 
$p\in \Gamma^{1}_{\rho}(\R^d\times J)$ 
with the assumption that
$$
\RE p(z;x,\xi) \geq c \la\xi\ra,
$$
for some positive constant~$c$. Assume that ~$w$ solves
\begin{equation*}
\partial_z w + T_p w =F,\quad w\arrowvert_{z=z_0}=w_0.
\end{equation*}
Then we have
\begin{equation*}
\lA w \rA_{X^r}\leq K\left\{\lA w_0\rA_{H^r}+
\lA F\rA_{Y^r}\right\},
\end{equation*}
for some positive constant~$K$ depending only on~$r,\rho,c$ and~$\mathcal{M}^1_\rho(p)$.
\end{prop}

In the following study, we will need the following Chemin-Lerner type of time-dependent spaces. See for example~\cite{BahouriCheminDanchinFourNL}.

\begin{defn}
If $s$ is any real number, and if $1\leq p,q,l\leq\infty$, we define the space $\tilde{L}^p_z(I;B^s_{q,l}(\R^d))$ as the space of 
tempered distributions $u$ such that
$$\lA u\rA_{\tilde{L}^p_z(I;B^s_{q,l}(\R^d))}:=\left(\sum_{k\in\N}2^{lks}\lA\Delta_ku\rA_{L^p_z(I;L^q)}^l\right)^\frac{1}{l}<+\infty.$$
\end{defn}
For $q=l=2$, $\tilde{L}^p_z(I;B^s_{2,2}(\R^d))=\tilde{L}^p_z(I;H^s(\R^d))$. Also for $q=l=\infty$, $\tilde{L}^p_z(I;B^s_{\infty,\infty}(\R^d))=\tilde{L}^p_z(I;C^s_*(\R^d))$.
As can be seen for example in \cite{BahouriCheminDanchinFourNL}, we have
$$\lA u\rA_{\tilde{L}^p_z(I;B^s_{q,l})}\leq\lA u\rA_{L^p_z(I;B^s_{q,l})}\text{ if }l\geq p,\qquad\lA u\rA_{\tilde{L}^p_z(I;B^s_{q,l})}\geq\lA u\rA_{L^p_z(I;B^s_{q,l})}\text{ if }l\leq p.$$
The paraproduct properties can be passed to the time dependent spaces as long as the exponents respect the conditions for H\"older inequality.

In the following, we will need the following Bernstein inequalities
\begin{lem}\label{lem:Berns}
Let $1\leq p\leq q\leq \infty, \alpha\in \N^d$. Then it holds that
\begin{gather*}
\lA\partial^\alpha S_k u\rA_{L^q}\leq C2^{kd(\frac{1}{p}-\frac{1}{q}+\la\alpha\ra)}\lA S_ku\rA_{L^p}\quad\text{ for }k\in\N,\\
\lA\Delta_ku\rA_{L^q}\leq C2^{kd(\frac{1}{p}-\frac{1}{q}+\la\alpha\ra)}\sup_{\la\beta\ra=\la\alpha\ra}\lA\partial^\beta\Delta_ku\rA_{L^p}\quad\text{ for }k\geq 1.
\end{gather*}
\end{lem}
We will also use the following parabolic smoothing effect
\begin{lem} \label{lem:parabsmooth}
Let $\kappa>0$ and $p\in [1,\infty]$. Then there exists some $c>0$ such that for any $t>0, k\geq 1$, we have
$$\lA e^{-t\kappa\langle D_x\rangle}\Delta_ku\rA_{L^p}\leq Ce^{-ct2^k}\lA\Delta_ku\rA_{L^p},$$
where~$\jap=(I-\Delta_x)^\mez$.
\end{lem}
The proof is classical (see \cite{BahouriCheminDanchinFourNL}).

\end{appendices}

\section*{Acknowledgments}
The author wishes to thank his advisor, Thomas Alazard, for proposing the subject and for many helpful discussions.

\bibliographystyle{plain}
\bibliography{ExplosionWW}

\end{document}